\documentclass[12pt]{amsart}

\usepackage[latin1]{inputenc}
\usepackage{amssymb,amsmath,amsfonts, amsthm}
\usepackage{enumerate}
\usepackage{mathrsfs}
\usepackage{times}
\usepackage{xcolor}
\usepackage{cite}

\usepackage{tikz}
\usepackage{3dplot}

\newtheorem{theorem}{Theorem}[section]
\newtheorem{lemma}[theorem]{Lemma}
\newtheorem{proposition}[theorem]{Proposition}
\newtheorem{corollary}[theorem]{Corollary}
\newtheorem{definition}[theorem]{Definition}
\numberwithin{equation}{section}

\newcommand{\N}{{\mathbb N}}
\newcommand{\Q}{\mathbb{Q}}
\newcommand{\R}{{\mathbb R}}

\newcommand{\cK}{{\mathcal K}} 
\newcommand{\cKo}{{\mathcal K}_{\scriptscriptstyle o}} 
\newcommand{\cKoi}{{\mathcal K}_{\scriptscriptstyle(o)}} 
\newcommand{\cP}{{\mathcal P}} 
\newcommand{\cPo}{{\mathcal P}_{\scriptscriptstyle o}} 
\newcommand{\cPoi}{{\mathcal P}_{\scriptscriptstyle(o)}} 

\newcommand{\cE}{{\mathcal E}}

\newcommand{\cN}{{\mathcal N}}

\newcommand{\cS}{{\mathcal S}}
\newcommand{\cT}{{\mathcal T}}

\DeclareMathOperator{\oZ}{\operatorname{Z}}

\DeclareMathOperator{\oY}{\operatorname{Y}}
\DeclareMathOperator{\oD}{\operatorname{D}}

\DeclareMathOperator{\otZ}{\operatorname{\tilde Z}}
\DeclareMathOperator{\obZ}{\operatorname{\bar Z}}
\newcommand{\oZZb}[2]{\operatorname{\overline{V}}_{#1,#2}}

\newcommand{\rn}{\R^n} 
\newcommand{\sn}{{\mathbb S}^{n-1}} 

\newcommand{\hm}{{\mathcal H}^{n-1}} 
\newcommand{\sd}{\delta_{\mathrm{s}}} 

\newcommand{\al}{\Omega} 
\newcommand{\epi}{\operatorname{epi}} 

\newcommand{\ind}{{\rm\bf I}} 
\newcommand{\dom}{\operatorname{dom}} 
\newcommand{\sq}{\mathbin{\vcenter{\hbox{\rule{.3ex}{.3ex}}}}} 
\newcommand{\infconv}{\mathbin{\Box}} 
\newcommand{\Hess}{{\operatorname{D}}^2} 
\newcommand{\MA}{\text{\rm MA}} 

\newcommand{\sln}{\operatorname{SL}(n)}
\newcommand{\slt}{\operatorname{SL}(2)}
\newcommand{\son}{\operatorname{SO}(n)}
\newcommand{\on}{\operatorname{O}(n)}
\newcommand{\SL}{\operatorname{SL}}

\newcommand{\sont}{{\operatorname{SO}(n)}\ltimes \R^n}
 
\newcommand{\funRn}{F(\R^n;\R)} 
\newcommand{\fconv}{{\mbox{\rm Conv}_{\text{\rm coe}}(\R^n)}} 
\newcommand{\fconvs}{{\mbox{\rm Conv}_{{\rm sc}}(\R^n)}} 
\newcommand{\fconvx}{{\mbox{\rm Conv}(\R^n)}} 
\newcommand{\fconvf}{{\mbox{\rm Conv}(\R^n; \R)}} 
\newcommand{\LC}{{\mbox{\rm LC}_{\text{\rm coe}}(\R^n)}}

\newcommand{\D}{\,\mathrm{d}}
\newcommand{\E}{\mathrm{e}}

\begin{document}
\title{Valuations on Convex Bodies and Functions}

\author{Monika Ludwig}
\address{Institut f\"ur Diskrete Mathematik und Geometrie,
Technische Universit\"at Wien,
Wiedner Hauptstra\ss e 8-10/1046,
1040 Wien, Austria}
\email{monika.ludwig@tuwien.ac.at}

\author{Fabian Mussnig}
\address{Institut f\"ur Diskrete Mathematik und Geometrie,
Technische Universit\"at Wien,
Wiedner Hauptstra\ss e 8-10/1046,
1040 Wien, Austria}
\email{fabian.mussnig@tuwien.ac.at}

\date{}

\begin{abstract} An introduction to geometric valuation theory is given. The focus is on
classification results for $\sln$ invariant and rigid motion invariant valuations on convex bodies and on convex functions.

\end{abstract}

\maketitle

\section{Introduction}
\label{sec:intro}
In his Third Problem, Hilbert asked whether, given any two polytopes of equal volume in $\R^3$, it is always possible to dissect the first into finitely many polytopes which can be reassembled to yield the second. In 1900, it was known that the answer to the corresponding question in $\R^2$ is \emph{yes}, but the question was open in higher dimensions. 

Let $\cP^n$ be the set of convex polytopes in $\R^n$. We say that $P\in \cP^n$ is \emph{dissected} into $P_1,\dots, P_m\in\cP^n$ and write
$P=P_1 \sqcup \cdots \sqcup P_m$,
if $P=P_1\cup \cdots\cup P_m$ and the polytopes $P_1, \dots, P_m$ have pairwise disjoint interiors. So, Hilbert's Third Problem asks whether for any $P, Q\in\cP^n$ of equal volume there are dissections 
\[P=P_1 \sqcup \cdots \sqcup P_m,\qquad Q=Q_1 \sqcup \cdots \sqcup Q_m,\]
and rigid motions $\phi_1, \dots, \phi_m$  such that 
\[P_i=\phi_i Q_i\] 
for $1\le i\le m$. We write $P\sim Q$ in this case.

We call a function $\oZ: \cP^n\to\R$ a \emph{valuation} if
\begin{equation*}
\oZ(P)+\oZ(Q)= \oZ(P\cup Q)+\oZ(P\cap Q)
\end{equation*}
for all $P,Q\in \cP^n$ with $P\cup Q\in\cP^n$ (and we set $\oZ(\varnothing):=0$). We call $\oZ$
\emph{simple} if $\oZ(P)=0$ for all polytopes that are not full-dimensional. We say that $\oZ$ is \emph{rigid motion invariant}  if 
\[\oZ(\phi P)= \oZ(P)\] 
for all rigid motions $\phi: \R^n\to\R^n$ and  $P\in \cP^n$.  If $\oZ:\cP^n\to \R$ is a simple, rigid motion invariant valuation, it is not difficult to see that $P\sim Q$ implies that $\oZ(P)=\oZ(Q)$.
Dehn \cite{Dehn1901} constructed a simple, rigid motion invariant valuation, now called Dehn invariant (see Section \ref{sec:rigid}), that is not a multiple of volume. He showed that the Dehn invariant of a regular simplex and a cube of the same volume do not coincide. Thereby he solved Hilbert's Third Problem and  showed that the answer to Hilbert's question is \emph{no} for  $n\ge 3$.

\goodbreak
Blaschke \cite{BlaschkeIntegralH2} took the critical next step by asking for  classification results for $G$ invariant valuations on $\cP^n$ and on the space of convex bodies, $\cK^n$, that is, of non-empty, compact, convex sets in $\R^n$, where $G$ is any group acting on $\R^n$. Blaschke's question is motivated by Felix Klein's Erlangen Program. 
We will discuss some of the results obtained in this tradition, in particular, focusing on the special linear group, $\sln$, and the group of rigid motions, $\sont$, where $\son$ is the group of (orientation preserving) rotations. Often additional regularity assumptions are required, and we consider continuous and upper semicontinuous valuations, where we equip $\cK^n$ and its subspaces with the topology induced by the Hausdorff metric.

In addition to classification results and their applications, structural results for spaces of valuations have attracted much attention in recent years. We refer to the books and surveys \cite{Alesker_Kent, AleskerFu_Barcelona, Bernig_AIG}. Valuations were also considered on various additional spaces, particularly on manifolds (see \cite{Alesker07}). 
Valuations with values in linear spaces and Abelian semigroups, including the space of convex bodies, were also studied (see \cite{Ludwig:ECM}).
We will restrict our attention to real-valued valuations defined on subspaces of $\cK^n$ and to recent results on valuations on spaces of real-valued functions. On a space $X$ of (extended) real-valued functions, a functional $\oZ:X\to\R$ is a \emph{valuation} if
\[\oZ(f)+\oZ(g)=\oZ(f\vee g) + \oZ(f \wedge g)\]
for all $f,g\in X$ such that also their pointwise maximum $f\vee g$ and  pointwise minimum $f\wedge g$ belong to $X$. 
Since we can embed spaces of convex bodies in various function spaces in such a way that unions and intersections of convex bodies correspond to pointwise minima and maxima of functions, this notion generalizes the classical notion.  We will discuss the results on valuations on convex functions.

\section{Basic Properties}
\label{sec:basic}
Let $\cS$ be a class of subsets of $\R^n$. We say that $\oZ:\cS\to \R$ is a \emph{valuation} if
\begin{equation*}\label{eq_val_def}
\oZ(P)+\oZ(Q)= \oZ(P\cup Q)+\oZ(P\cap Q)
\end{equation*}
for all $P,Q\in\cS$ such that $P\cap Q, P\cup Q\in\cS$, and $\oZ(\varnothing)=0$. Given a Borel measure on $\R^n$,  its restriction to $\cK^n$ is clearly a valuation. So, in particular, $n$-dimensional Lebesgue measure, $V_n$, induces a valuation on $\cK^n$. As we will see, there are important valuations that are not induced by measures.

\goodbreak
Let $\cS$ be \emph{intersectional}, that is, if $P,Q\in\cS$, then $P\cap Q\in\cS$. We say that $\oZ: \cS\to \R$ satisfies the \emph{inclusion-exclusion principle} on $\cS$   if 
\begin{equation}\label{eq_in_ex}
\oZ(P_1\cup\dots\cup P_m)=\sum_{\varnothing\neq
      J\subset\{1,\ldots,m\}}(-1)^{\vert J\vert -1} \oZ(P_J)
  \end{equation}
for $P_1, \dots, P_m\in \cS$ and $m\ge1$  whenever $P_1\cup  \cdots\cup P_m\in \cS$. Here $P_J:= \bigcap_{j\in J} P_j$ and $\vert J\vert$ is the cardinality of the set $J$. The inclusion-exclusion principle holds for every valuation on $\cP^n$ and every continuous valuation on $\cK^n$ (see \cite{Klain:Rota, Schneider:CB2}). If $\oZ:\cP^n\to \R$ is, in addition, simple, we have
\begin{equation}\label{eq_finite_add}
\oZ(P_1\sqcup\cdots\sqcup P_m)=\oZ(P_1)+\dots+\oZ(P_m)
\end{equation}
for $P_1, \dots, P_m\in\cP^n$.

For $K,L\in\cK^n$, define the \emph{Minkowski sum} by
\[K+L:=\{x+y: x\in K, y\in L\}.\]
The following lemma describes a way to obtain new valuations from a given one.

\begin{lemma} \label{lem_val_mink_add}
Let $\oZ:\cK^n\to \R$ be a valuation. If $C \in \cK^n$ is a fixed convex body and 
\[
\oZ_{C}(K):=\oZ(K+C),
\]
for $K \in \cK^n$, then $\oZ_{C}$ is a valuation on $\cK^n$.
\end{lemma}

\begin{proof}
The following statement is easily seen to hold for subsets $C,K,L\subset \R^n$,
\begin{equation}\label{eq_union}
K\cup L+C=(K+C)\cup(L+C).
\end{equation}
Now, let $C, K,L\in\cK^n$ be such that $K\cup L\in\cK^n$. If $x \in(K+C) \cap(L+C)$, then $x=y+c=z+d$ with $y \in K, z \in L$ and $c, d \in C$. Since $K \cup L$ is convex, there is $t \in[0,1]$ such that $(1-t) y+t z \in K \cap L$ and hence
\begin{equation*}
x =(1-t)(y+c)+t(z+d) 
=(1-t) y+t z+(1-t) c+t d. 
\end{equation*}
Thus $(K+C) \cap(L+C) \subset(K \cap L)+C$. 

\goodbreak
Since it is easy to see that
$(K\cap L)+C\subset (K+C)\cap (L+C)$,
it follows that
\begin{equation}\label{eq_intersection}
(K+C) \cap(L+C) =(K \cap L)+C.
\end{equation}
Applying $\oZ$ to \eqref{eq_union} and to \eqref{eq_intersection} for convex bodies $C, K, L$ and adding, we obtain the statement.
\end{proof}

For $p\in\R$, a functional $\oZ:\cK^n\to \R$ is called \emph{homogeneous of degree $p$} (or \emph{$p$-homogeneous}), if 
\[\oZ(t\,K)= t^p\oZ(K)\] 
for $t>0$ and $K\in\cK^n$.  A functional $\oZ:\cK^n\to \R$ is \emph{increasing} if $K\subset L$ implies that $\oZ(K)\le \oZ(L)$. We will also use corresponding definitions for subsets of $\cK^n$.

The $n$-dimensional volume $V_n: \cK^n \to [0,\infty)$ is a valuation. Lemma \ref{lem_val_mink_add} implies that also $K\mapsto V_n(K+ r\,B^n)$ is a valuation on $\cK^n$ for $r\ge 0$, where $B^n$ is the $n$-dimensional unit ball. Therefore, it follows from the Steiner formula,
\begin{equation}
\label{eq_steiner}
V_n({K+r B^n})=\sum_{j=0}^n r^{n-j}\kappa_{n-j} V_j(K),
\end{equation}
where $r\ge 0$ and $\kappa_j$ is the $j$-dimensional volume of the unit ball in $\R^j$ (with the convention that $\kappa_0=1$), that all intrinsic volumes $V_0, \dots, V_n$ are valuations on $\cK^n$. Recall that all intrinsic volumes are continuous and increasing functionals on $\cK^n$ and that $V_0$ is the Euler characteristic and $V_0(K)=1$ for all $K\in\cK^n$. Also, recall that $V_j(K)$ is the $j$-dimensional volume of $K$ if $K$ is contained in a $j$-dimensional plane and that $V_j$ is $j$-homogeneous.

We will use the following notation. Let $e_1, \dots, e_n$ be the vectors of the canonical basis of $\R^n$. For $x,y\in\R^n$, we write $\langle x,y\rangle$ for the inner product and $\vert x\vert$ for the Euclidean norm of $x$. The convex hull of subsets $A_1, \dots, A_m\subset\R^n$ is written as $[A_1, \dots, A_m]$ and the convex hull of $x_1, \dots, x_m\in\R^n$ as $[x_1, \dots, x_m]$.  If $E\subset \R^n$ is an affine plane in $\R^n$, then $\cK(E)$ and $\cP(E)$ are the sets of convex bodies and convex polytopes, respectively, contained in $E$.

\section{$\sln$ Invariant Valuations}

Blaschke \cite{BlaschkeIntegralH2}  obtained the first classification theorem of invariant valuations on $\cK^n$.

\begin{theorem}[Blaschke]
A functional $\oZ:\cK^n\to \R$ is  a continuous, translation and $\sln$ invariant valuation if and only if  there are constants $c_0, c_n\in\R$ such that 
\[\oZ(K) = c_0 V_0(K)+c_n V_n(K)\]
for every $K\in\cK^n$.
\label{thm_Blaschke}
\end{theorem}

\noindent
In the next section, we will obtain a complete classification of translation invariant valuations in the one-dimensional case, and in the following section, a complete classification of translation and $\sln$ invariant valuations on convex polytopes. Here, no assumptions on the continuity of the valuation are needed. Theorem \ref{thm_Blaschke} will be a simple consequence. The situation is different for valuations on convex bodies, where additional (non-continuous) valuations exist that vanish on convex polytopes. We will describe some of these valuations in Section \ref{sec:affine}.

\subsection{The One-dimensional Case}

We call a function $\zeta: [0,\infty)\to \R$ a \emph{Cauchy function} if it is a solution to the \emph{Cauchy functional equation}, that is,
\[\zeta(x+y)=\zeta(x)+\zeta(y)\]
for every $x,y\in[0,\infty)$. Cauchy functions are well understood and can be completely described (if we assume the axiom of choice) by their values on a Hamel basis. 

\goodbreak
\begin{proposition}\label{prop_trans1}
A functional $\oZ:\cP^1\to \R$ is  a translation invariant valuation if and only if  there are a constant $c_0\in\R$  and a Cauchy function $\zeta:[0,\infty)\to \R$ such that 
\[\oZ(P) = c_0 \,V_0(P)+ \zeta\big(V_1(P)\big)\]
for every $P\in\cP^1$. 
\end{proposition}

\begin{proof}
Set $c_0:=\oZ(\{0\})$ and define $\otZ: \cP^1\to \R$ by
\[\otZ(P):=\oZ(P)-c_0 V_0(P).\]
Note that $\otZ$ is a simple, translation invariant valuation on $\cP^1$. Define $\zeta:[0,\infty)\to\R$ by setting
\[\zeta(x):= \otZ([0,x]).\]
Since $\otZ$ is a simple, translation invariant valuation,
\[\zeta(x+y)= \otZ([0,x+y))=\otZ([0,x])+\otZ([x,x+y])=\zeta(x)+\zeta(y)\]
for every $x,y\in[0,\infty)$. Hence $\zeta$ is a Cauchy function. Using that $\otZ$ is translation invariant, we get $\otZ(P)=\zeta(V_1(P))$ for $P\in\cP^1$, which concludes the proof.
\end{proof}

Since every continuous Cauchy function is linear, we obtain the following result.

\begin{corollary}\label{cor_one}
A functional $\oZ:\cP^1\to \R$ is  a continuous and translation invariant valuation if and only if  there are constants $c_0, c_1\in\R$ such that 
\[\oZ(P) = c_0 V_0(P) + c_1 V_1(P)\]
for every $P\in\cP^1$. 
\end{corollary}

\noindent
A corresponding classification result holds for upper semicontinuous and translation invariant valuations on $\cP^1$. Such a result also holds for Borel measurable and translation invariant valuations on $\cP^1$, since every Borel measurable Cauchy function is linear.

\goodbreak
\subsection{$\sln$ Invariant Valuations on Convex Polytopes}

The following result gives a complete classification of translation and $\sln$ invariant valuations on polytopes.

\begin{theorem}
A functional $\oZ\colon\cP^n\to \R$ is  a translation and $\sln$ invariant valuation if and only if  there are a constant $c_0\in\R$  and a Cauchy function $\zeta\colon[0,\infty)\to \R$ such that 
\[\oZ(P) = c_0 V_0(P)+ \zeta\big(V_n(P)\big)\]
for every $P\in\cP^n$. \label{thm_polytopes}
\end{theorem}

\begin{proof}
Set $c_0:=\oZ(\{0\})$ and define $\otZ: \cP^n\to \R$ by
\[\otZ(P):=\oZ(P)-c_0 V_0(P).\]
Note that $\otZ$ is a translation invariant valuation on $\cP^n$ that vanishes on singletons, that is, sets of the form $\{x\}$ with $x\in\R^n$. We show that there is a Cauchy function $\zeta: [0,\infty)\to\R$  such that
\begin{equation}\label{eq_proof}
\otZ(P)= \zeta(V_n(P))
\end{equation}
for every $P\in\cP^n$.

\goodbreak
We use induction on the dimension $n$. By Proposition \ref{prop_trans1}, the statement \eqref{eq_proof} is true for $n=1$. Let $n\ge2$. Assume that it is true for valuations on $\cP^{n-1}$. Hence it is also true for valuations on $\cP(E)$ with $E$ any hyperplane in $\R^n$. The induction assumption implies that there is a Cauchy function $\tilde\zeta:[0,\infty)\to \R$ such that
\[\otZ(P)=\tilde\zeta(V_{n-1}(P))\]
for every $P\in\cP(E)$.  Note that the invariance properties of $\otZ$ imply that $\tilde\zeta$ does not depend on $E$. Since $\otZ$ vanishes on singletons, we have $\tilde\zeta(0)=0$. Let $E$ be spanned by the first $(n-1)$ basis vectors $e_1, \dots, e_{n-1}$ and define $\phi\in\sln$ by   setting $\phi e_1= t\, e_1$ with $t>0$ and $\phi e_j=e_j$ for $1<j<n$ and $\phi e_n=\frac1t e_n$. Since $\phi E=E$, it follows from the $\sln$ invariance of $\otZ$ that
\[\tilde\zeta(t)=\otZ(\phi [0,1]^{n-1})=\otZ([0,1]^{n-1})= \tilde\zeta(1)\]
for every $t>0$. 
This implies that $\tilde\zeta \equiv0$ and shows that $\otZ$ is simple. Thus it suffices to show that \eqref{eq_proof} holds for every  simple, translation and $\sln$ invariant valuation $\otZ:\cP^n\to \R$.

\goodbreak
Define $\zeta:[0,\infty)\to\R$ by setting
\[\zeta(s):= \otZ(\sqrt[n]{s\,n!}\, [0, e_1, \dots, e_n])\]
and note that 
\[\otZ(S)= \zeta(V_n(S))\]
for every simplex $S\in\cP^n$, as the valuation $\otZ$ is simple, translation and $\sln$ invariant and every $n$-dimensional simplex is a translate  of an $\sln$ image of the simplex $\sqrt[n]{s\,n!}\,[0, e_1, \dots, e_n]$ for some $s>0$. For $0<r<1$, we dissect the $n$-dimensional simplex with vertices $v_0, \dots, v_n\in\R^n$ into the $n$-dimensional simplices $T_1$ with vertices $v_0, r \,v_0+(1-r) v_1, v_2, \dots, v_n$ and $T_2$ with the vertices $(1-r)v_0 +r\,v_1, v_1, v_2, \dots, v_n$. 
Since $\otZ$ is a simple valuation,
\begin{equation}\label{eq_simplex_valu}
\otZ(T_1\cup T_2)=\otZ(T_1)+\otZ(T_2).
\end{equation}

\begin{figure}
	\centering
	\tdplotsetmaincoords{80}{35}
	\begin{tikzpicture}[scale=4.5,line join=bevel,tdplot_main_coords]
		\tikzset{xzplane/.style={canvas is xz plane at y=#1}}
		\tikzset{yzplane/.style={canvas is yz plane at x=#1}}
		\tikzset{xyplane/.style={canvas is xy plane at z=#1}}
		
		\pgfmathsetmacro\arr{0.5}
		
		\draw [thick,densely dashed] (0,0,0) -- (0,2.5,0);
		\draw [thick] (0,0,0) -- (2,0,0) -- (0.2,0,0.9) -- (0,0,0);
		\draw [thick] (2,0,0) -- (0,2.5,0) -- (0.2,0,0.9);
		
		\draw (0,0,0) node[anchor=north]{\fontsize{10}{10}\selectfont $v_0$};
		\draw (2,0,0) node[anchor=north]{\fontsize{10}{10}\selectfont $v_1$};
		\draw (0,2.5,0) node[anchor=south]{\fontsize{10}{10}\selectfont $v_2$};
		\draw (0.2,0,0.9) node[anchor=east]{\fontsize{10}{10}\selectfont $v_3$};
	
		\draw (2*\arr,0,0) -- (0.2,0,0.9);
		\draw [dashed] (2*\arr,0,0) -- (0,2.5,0);
		
		\draw (2*\arr,0,0) node[anchor=north]{\fontsize{10}{10}\selectfont $(1-r)v_0+ r v_1$};
	\end{tikzpicture}
	\caption{Decomposition of $[v_0,\ldots,v_n]$ into $T_1$ and $T_2$.}
\end{figure}
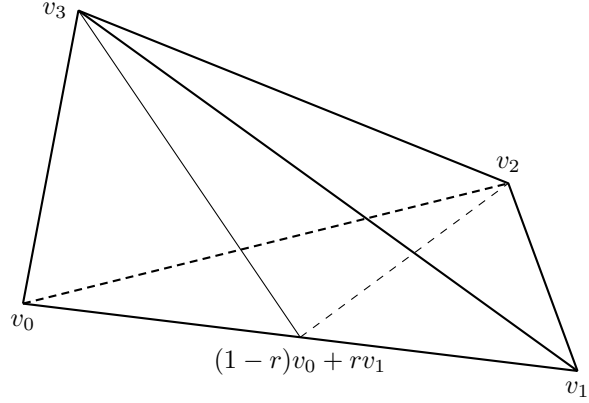

Choosing $v_0:=0$ and $v_j:=\sqrt[n]{(s+t)n!}\,e_j$ for $j=1, \dots, n$ as well as $r:=s/(s+t)$, we obtain from \eqref{eq_simplex_valu} that
\[\zeta(s+t)= \zeta(s)+\zeta(t)\]
for every $s,t\in(0,\infty)$. Hence, $\zeta$ is a Cauchy function. By \eqref{eq_finite_add} and since we can dissect every polytope into simplices, we conclude that \eqref{eq_proof} holds for every $P\in\cP^n$.
\end{proof}

\goodbreak
Properties of Cauchy functions immediately give the following result, which, in turn, implies Theorem~\ref{thm_Blaschke}.

\begin{corollary}
A functional $\oZ:\cP^n\to \R$ is  a continuous, translation and $\sln$ invariant valuation if and only if  there are constants $c_0, c_n\in\R$ such that 
\[\oZ(P) = c_0 V_0(P)+c_n V_n(P)\]
for every $P\in\cP^n$. 
\end{corollary}

\noindent
Corresponding statements  hold for upper semicontinuous valuations and for Borel measurable valuations.

\goodbreak
We remark that classification results for $\sln$ invariant valuations are also known without assuming translation invariance (see \cite{LudwigReitzner_AP}). In particular, the following result holds. Let $\cPo^n$ be the space of convex polytopes containing the origin. 

\begin{theorem}
\label{thm:blaschke_pon}
A functional $\oZ:\cPo^n\to \R$ is  a continuous, $\sln$ invariant valuation if and only if  there are constants $c_0, c_n\in\R$ such that 
\[\oZ(P) = c_0 V_0(P)+c_n V_n(P)\]
for every $P\in\cPo^n$. 
\end{theorem}

\noindent
On $\cP^n$, there are additional $\sln$ invariant valuations. In particular, $P\mapsto V_n([0,P])$ is such a valuation (see \cite{LudwigReitzner_AP} for a complete classification). On $\cPoi^n$, the space of convex polytopes containing the origin in their interiors, $P\mapsto V_n(P^\circ)$, the functional that associates with $P$ the volume of its polar body, is an $\sln$ variant valuation. A complete classification of $\sln$ invariant valuations on $\cPoi^n$ was established by Haberl and Parapatits \cite{Haberl:Parapatits_centro}.

\subsection{Affine Surface Area}\label{sec:affine}

While we have established a complete classification of translation and $\sln$ invariant valuations on $\cP^n$, such a result is not known on $\cK^n$, and there are additional valuations on $\cK^n$ that vanish on $\cP^n$. The classical \emph{affine surface area} $\Omega: \cK^n\to \R$ is such a valuation. It is defined by
\begin{equation}\label{eq_affine surface area}
\Omega(K)= \int_{\partial K} \kappa(K,x)^{\frac1{n+1}}\,\D \hm(x),
\end{equation}
where $\kappa(K,x)$ is the generalized Gaussian curvature of $\partial K$ at $x$ and integration is with respect to the $(n-1)$-dimensional Hausdorff measure $\hm$ on the boundary, $\partial K$, of $K$.  By a classical result of Aleksandrov, the boundary of a convex body is twice differentiable almost everywhere and
hence $\kappa(K,x)$ is defined almost everywhere and it can be shown that $x\mapsto \kappa(K,x)$ is measurable. 
We remark that the generalized Gaussian curvature is the density of the absolutely continuous part of the curvature measure $C_0(K,\cdot)$, 
where 
$C_0(K,B):= \hm(\nu_K(B))$ for a Borel set $B\subset\partial K$, and $\nu_K$ is the spherical image map that assigns to $x\in\partial K$ the set  of all unit normal vectors of supporting hyperplanes of $K$ containing $x$ (see \cite[Chapter 4]{Schneider:CB2}).
Hence
\[\int_{\partial K} \kappa(K,x)\, \D\hm (x) \le \hm (\partial B^n)= n\kappa_n\]
for every $K\in\cK^n$, and
by Jensen's inequality,
\begin{equation}\label{eq_upper}
\int_{\partial K}\kappa(K,x)^{\frac1{n+1}}\,\D \hm(x)\le \big(n\kappa_n\big)^{\frac1{n+1}} 
\big(\int_{\partial K}  \D \hm(x)\big)^\frac{n}{n+1}.
\end{equation}
This implies that the integral in \eqref{eq_affine surface area} is finite for every  $K\in\cK^n$. 

\goodbreak
The definition of affine surface area for convex bodies with smooth boundary is classical and goes back to Blaschke and Pick  \cite{Blaschke}. They established that $\Omega$ is \emph{equi-affine invariant}, that is, $\Omega$ is translation and $\sln$ invariant.
The extension to general convex bodies is more recent and due to Leichtwei\ss \ \cite{Leichtweiss86b}, Lutwak \cite{Lutwak91} and Sch\"utt and Werner \cite{Schuett:Werner90}.  Lutwak \cite{Lutwak91} proved that $\Omega$ is upper semicontinuous on $\cK^n$, that is, for every sequence of convex bodies $K_j$ converging to a convex body $K$, we have
\[\Omega(K)\ge \limsup_{j\to\infty} \Omega(K_j).\]
It follows from \eqref{eq_affine surface area} that $\Omega$ vanishes on polytopes and is therefore not continuous. The valuation property of $\Omega$ on $\cK^n$ follows directly from \eqref{eq_affine surface area}.

\goodbreak
Note that $\Omega$ is translation invariant and that $\Omega(K)=0$ if $K$ is lower dimensional. Hence we may assume in the following that the origin is an interior point of $K$.
Clearly, \eqref{eq_affine surface area} can be rewritten as
\begin{equation}\label{eq_affine surface area2}
\Omega(K)= \int_{\partial K} \kappa_0(K,x)^{\frac1{n+1}}\,\D V_K(x),
\end{equation}
where 
\[\kappa_0(K,x):= \frac{\kappa(K,x)}{\langle x, n_K(x)\rangle^{n+1}}\]
and
\[\D V_K(x):= \langle x, n_K(x)\rangle \,\D \hm(x).\]
Here, $n_K(x)$ is the unit outer normal vector of $K$ at $x$, which is uniquely defined almost everywhere on $\partial K$, and $\langle x, n_K(x)\rangle$ is the distance to the origin of the tangent hyperplane to $K$ at such $x$. In \eqref{eq_affine surface area2}, it is easy to see that $\Omega$ is $\sln$ invariant. Indeed, for a Borel set $B\subset \partial K$, using the fact that the volume of a cone is the product of its height divided by $n$ and the $(n-1)$-dimensional volume of its base, we see that  $\tfrac1n V_K(B)$ is just the $n$-dimensional volume of the set $\{ t\, B: t\in [0,1]\}$. Consequently, 
\[V_{\phi K}(\phi B)= V_K(B)\]
for every $\phi\in\sln$ and every Borel set $B\subset \partial K$. Moreover,
\begin{equation*}
\kappa_0(\phi K,\phi x)= \kappa_0(K,x)
\end{equation*}
for every $\phi\in\sln$ and every $x\in \partial K$ where $\kappa_0(K,x)>0$. This is a simple consequence of the following geometric interpretation of $\kappa_0(K,x)$,
\begin{equation*}
\kappa_0(K,x)= \frac{\kappa_n^2}{V_n(E_K(x))^2},
\end{equation*}
where  $E_K(x)$ is the unique centered ellipsoid that osculates $K$ at $x$. We remark that 
\begin{equation}\label{eq_Orlicz}
K \mapsto \int_{\partial K} \zeta(\kappa_0(K,x))\,\D  V_K(x)
\end{equation}
is an $\sln$ invariant valuation on $\cKoi^n$, the set of convex bodies containing the origin in their interiors when $\zeta:[0,\infty)\to [0,\infty)$ is a suitable continuous function. The functionals defined in \eqref{eq_Orlicz} are called \emph{Orlicz affine surface areas}. If $\zeta(t):=t^p$  for $t>0$ with $p>-n$, the so-called \emph{$L_p$ affine surface area} of $K$ is obtained, which was introduced by Lutwak \cite{Lutwak96}. Classification results for $\sln$ invariant valuations on $\cKoi^n$ were established in \cite{Haberl:Parapatits_centro,Ludwig:origin,Ludwig:Reitzner2} and characterizations of $L_p$ and Orlicz affine surface areas in \cite{Ludwig:Reitzner2}.

\goodbreak
The following result  from  \cite{Ludwig:affinelength,Ludwig:Reitzner} strengthens  Theorem \ref{thm_Blaschke} and establishes a characterization of affine surface area.

\goodbreak
\begin{theorem}
A functional $\oZ:\cK^n\to \R$ is  an upper semicontinuous, translation and $\sln$ invariant valuation if and only if  there are constants $c_0, c_n\in\R$  and $c\ge 0$  such that 
\[\oZ(K)= c_0 V_0(K) +c_n V_n(K)+ c\,\Omega(K)\]
for every $K\in\cK^n$. \label{thm_asa}
\end{theorem}

We present the proof of Theorem \ref{thm_asa} in the case $n=2$ from \cite{Ludwig:affinelength}. We call a closed triangle $T=T(x,y)$ a \emph{support triangle} of $K\in\cK^2$ with \emph{endpoints} $x$ and $y$, if $x,y\in\partial K$ and $T$ is bounded by \emph{support lines} (that is, 1-dimensional support hyperplanes) to $K$ at $x$ and $y$ and the chord connecting $x$ and $y$.
\begin{figure}[htp]
	\centering
	\begin{tikzpicture}[scale=4,rotate=25]
	
	    \draw[thick,rotate=30,fill=lightgray] (0,0) ellipse (0.595 and 0.788375);
	        
	    \draw[black,fill=black]  (0.105,-0.745) circle(0.01);
	    \draw[dashed] (-0.375,-0.6925)--(0.745,-0.815);
	    \draw (0.105,-0.745)--(0.745,-0.815);
	    \draw (0.105,-0.745) node[anchor=north]{\fontsize{10}{10}\selectfont $x$};
	
	    \draw[black,fill=black]  (0.64,-0.05) circle(0.01);
	    \draw[dashed] (0.56125,0.52375)--(0.745,-0.815);
	    \draw (0.64,-0.05)--(0.745,-0.815);
	    \draw (0.64,-0.05) node[anchor=west]{\fontsize{10}{10}\selectfont $y$};
	    
	    \draw (0.105,-0.745)--(0.64,-0.05);
	
    	\end{tikzpicture}
	\caption{Support triangle of a convex body $K\in\cK^2$ with endpoints $x,y\in\partial K$.}
\end{figure}
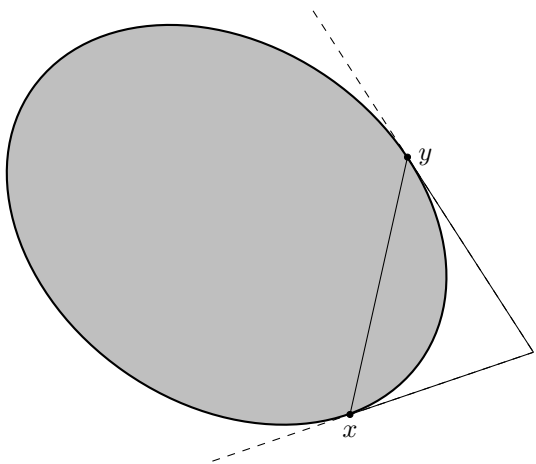

A \emph{cap} of a convex body $K$ is the intersection of a closed half-space and $K$. We set
$\sd(K,L):= V_2(K\triangle L)$ for $K,L\in \cK^2$, where $K\triangle L:=(K\cup L)\backslash(K\cap L)$ is the symmetric difference of $K$ and $L$. Note that the \emph{symmetric difference metric} $\sd$ induces on full-dimensional convex bodies the same topology  as the Hausdorff metric. 

\goodbreak
We require the following lemma, whose proof is omitted as it is very similar to the proof of Proposition \ref{prop_trans1}.

\begin{lemma}
If ${\,\oZ:\cK^2\to\R}$ is an upper semicontinuous,  rotation invariant valuation  that vanishes on polytopes, then  
\[\oZ(C)=c\,\al(C)\]
for every cap $C$ of $B^2$, where $c:= \oZ(B^2)/\al(B^2)$.\label{lem_sector}
\end{lemma}

Let $\cE^2$ be the family of all convex bodies in $\R^2$ which may be dissected into finitely many polytopes and caps of unit ellipses. Here, any equi-affine image of the two-dimensional unit ball $B^2$ is called a unit ellipse.  Since planar polytopes  belong to $\cE^2$, the set $\cE^2$ is dense in $\cK^2$.

\begin{proposition}
If $\,\oZ:\cK^2\to[0,\infty)$ is an upper semicontinuous, translation and $\,\slt$ invariant valuation  that vanishes on polytopes, then 
\[\oZ(K)= \sup\{ \limsup_{k\to\infty} \oZ(E_k): E_k\to K, E_k\in\cE^2\}\]
for every $K\in\cK^2$.\label{prop_unique}
\end{proposition}

\begin{proof}
Since $\oZ$ is upper semicontinuous, we have 
\[\oZ(K)\ge  \limsup_{k\to\infty} \oZ(E_k)\]
for every $K\in\cK^2$ and for every sequence $E_k\in \cE^2$ such that $E_k\to K$.
To prove the statement of the proposition, assume on the contrary that there is $K\in\cK^2$  such that
\begin{equation}\label{eq_contra}
\oZ(K)>\limsup_{k\to\infty}\oZ(E_k)
\end{equation}
for all sequences $E_k$ with $E_k\in\cE^2$ and $E_k\to K$. By \eqref{eq_upper},  the affine surface area of $E$ is uniformly bounded for all convex bodies $E$ with $\sd(K, E)< 1$, say.  Therefore, by \eqref{eq_contra}, for every $\varepsilon>0$ small enough, there is $0<\delta<1$ such that
\begin{equation}\label{eq_echt}
\oZ(K)\ge \oZ(E)+\varepsilon\, \al(E)
\end{equation}
for every $E\in\cE^2$ with $\sd(K, E)<\delta$.

\goodbreak
We approximate $B^2$ by a sequence of convex bodies built from 
suitable pieces of $K$ and  show that \eqref{eq_echt} leads to a contradiction.
Without loss of generality, assume that the origin is an interior point of $K$.
Choose $k$ rays starting at the origin such that
\begin{equation}\label{eq_eps}
\sum_{i=1}^{k} V_2(T_i^{(k)})<\delta
\end{equation}
where  $T_i^{\scriptscriptstyle(k)}=T(x_i^{\scriptscriptstyle(k)},x_{i+1}^{\scriptscriptstyle(k)}) $ are support triangles and $x_1^{\scriptscriptstyle(k)}, \dots, x_k^{\scriptscriptstyle(k)}, x_{k+1}^{\scriptscriptstyle(k)}=x_1^{\scriptscriptstyle(k)}$ are the consecutive points where the rays intersect $\partial K$.  
For every $T_i^{\scriptscriptstyle(k)}$ with non-empty interior, there is a unique arc of the unit ellipse which 
touches the two sides of $T_i^{\scriptscriptstyle(k)}$ which are given by the support lines
of $K$. We denote by $E_i^{\scriptscriptstyle(k)}$ the convex body bounded by this arc of an ellipse and the chord  connecting $x_i^{\scriptscriptstyle(k)}$ and $x_{i+1}^{\scriptscriptstyle(k)}$.
In the case that $T_i^{\scriptscriptstyle(k)}$  has empty interior, we set $E_i^{\scriptscriptstyle(k)}:=T_i^{\scriptscriptstyle(k)}$.

\goodbreak
We define 
    \[ E_k:= \bigcup_{i=1}^k E_i^{(k)} \cup \Big (K\backslash \bigcup_{i=1}^{k} T_i^{(k)}\Big)\]
Note that  $E_k\in \cE^2$ and that \eqref{eq_eps} implies that 
$\sd(K, E_k)<\delta$. 

\goodbreak
Since $\oZ$ and $\al$  vanish on polytopes,  \eqref{eq_echt}  implies that
\begin{equation*}
    \sum_{i=1}^{k}\oZ(K\cap T_i^{(k)})= \oZ(K)\ge \oZ(E_k)+\varepsilon\,\al(E_k)=\sum_{i=1}^{k}\big(\oZ(E_k\cap T_i^{(k)})+ \varepsilon\,\al(E_k\cap T_i^{(k)})\big).
\end{equation*}
Consequently, for every $k$, there exists a support triangle  $T_{i_k}^{\scriptscriptstyle(k)}$ with non-empty interior such that 
\begin{equation}\label{eq_ungl}
    \oZ(K\cap T_{i_k}^{(k)})\ge \oZ(E_k\cap T_{i_k}^{(k)}) +\varepsilon\,\al(E_k\cap T_{i_k}^{(k)}).
\end{equation}

We take  an equi-affine transformation  $\phi^{\scriptscriptstyle(k)}$ which 
transforms  $T_{i_k}^{\scriptscriptstyle(k)}$ into a support triangle $\tilde T^{\scriptscriptstyle(k)}$ of  $B^2$, and denote by $\tilde C^{\scriptscriptstyle(k)}$ and $B^{\scriptscriptstyle(k)}$  the images under  
$\phi^{\scriptscriptstyle(k)}$  of the caps $K\cap T_{i_k}^{\scriptscriptstyle(k)}$ and $E_k\cap T_{i_k}^{\scriptscriptstyle(k)}$, respectively. By \eqref{eq_ungl} and the equi-affine invariance of $\oZ$, we have
\begin{equation}\label{eq_gross}
    \oZ(\tilde C^{(k)})\ge \oZ(B^{(k)})+\varepsilon\,\al(B^{(k)}).
\end{equation}
Let $l_k$ be the largest integer such that there are rotations 
$\psi_1,\ldots,\psi_{l_k}$ with the property that 
$\psi_1(\tilde T^{\scriptscriptstyle(k)}), \ldots$, $ \psi_{l_k}(\tilde T^{\scriptscriptstyle(k)})$ are non-overlapping support
triangles of $B^2$. Since for a sector of $B^2$ with an angle $2 \,\alpha$ at the origin, the area of a support triangle to $B^2$ is
$\sin^2\alpha \tan \alpha$,
we have
\begin{equation}\label{eq_number} 
    \sin^2\left( \frac{\pi}{l_k+1}\right) \tan \left( \frac \pi {l_k+1}\right)\le V_2(\tilde T^{(k)})\le \sin^2\left( \frac \pi {l_k} \right)\tan \left( \frac \pi {l_k}\right).
\end{equation}
We construct convex bodies 
    \[ \tilde K_k :=\bigcup_{i=1}^{l_k} \psi_i(\tilde C^{(k)}) \cup \Big(B^2\backslash \bigcup_{i=1}^{l_k} \psi_i(\tilde T^{(k)})\Big).\]
Note that \eqref{eq_gross} implies that 
\begin{equation}\label{eq_CnEn}
    \oZ(\tilde K_k)\ge \oZ(B^2)+\frac \varepsilon2\, \al(B^2)
\end{equation}
for $k$ sufficiently large. Since $\sd(\tilde K_k, B^2)\le l_k \,V_2(\tilde T^{\scriptscriptstyle(k)})$, it follows from \eqref{eq_number} that
\begin{equation}\label{eq_conv} 
\tilde K_k\to B^2 
\end{equation}
as $k\to\infty$.
Thus by the upper semicontinuity of $\oZ$, by \eqref{eq_conv},  \eqref{eq_CnEn} and \eqref{eq_equal_ellipse}, we obtain that
    \[\oZ(B^2)\ge \limsup_{k\to\infty}\oZ(\tilde K_k)\ge \oZ(B^2) +\frac{\varepsilon}2\,\al(B^2).
        \]
This is a contradiction since $\varepsilon>0$ and $\al(B^2)>0$, which concludes the proof of the proposition.
\end{proof}

Note that we can apply Proposition \ref{prop_unique} with $\oZ=\al$ and obtain that
\begin{equation}\label{eq_al_char}
    \al(K)= \sup\{ \limsup_{k\to\infty} \al(E_k): E_k\to K, E_k\in\cE^2\}
\end{equation}
for every $K\in\cK^2$.

\begin{proof}[Proof of Theorem \ref{thm_asa}\, for $n=2$]
Let $\oZ:\cK^2\to \R$ be  an upper semicontinuous, translation and $\slt$ invariant valuation.
By Theorem \ref{thm_polytopes} and since upper semicontinuous Cauchy functions are linear, there are $c_0, c_2\in\R$ such that
\[\oZ(P)=c_0 V_0(P)+ c_2 V_2(P)\]
for every $P\in \cP^2$. Define $\otZ: \cK^2 \to \R$ by
\[\otZ(K):= \oZ(K)- c_0 V_0(K)-c_2V_2(K)\]
and note that $\otZ$ is an upper semicontinuous, translation and $\slt$ invariant valuation that vanishes on polytopes. 
For every $K\in\cK^2$, there is a sequence of polytopes $P_k$ with $P_k\to K$. Hence, the upper semicontinuity of $\otZ$ implies that  
\[\otZ(K)\ge \limsup_{k\to\infty} \otZ(P_k)= 0,\]
which shows that $\otZ$ is non-negative. 
Using Lemma \ref{lem_sector} and the translation and $\slt$ invariance of $\otZ$, we see that 
\[\otZ(C)=c\,\al(C)\]
for every cap $C$ of a unit ellipse, where $c= \oZ(B^2)/\al(B^2)$. 
Since $\oZ$ vanishes on polytopes, it is a simple valuation, and it follows from \eqref{eq_finite_add} that
\begin{equation}\label{eq_equal_ellipse}
   \otZ(E)= c\,\al(E) 
\end{equation}
for $E\in\cE^2$. 
Proposition \ref{prop_unique} and \eqref{eq_al_char} now complete the proof of the theorem.
\end{proof}
\goodbreak

For $K\in\cK^2$, Blaschke \cite{Blaschke} gave the following definition of affine surface area.
Choose  sub\-division points $x_1^{\scriptscriptstyle(k)},\ldots, x_k^{\scriptscriptstyle(k)}, x_{k+1}^{\scriptscriptstyle (k)}=x_1^{\scriptscriptstyle(k)}$ on $\partial K$ and support triangles $T_1^{\scriptscriptstyle (k)},\ldots, T_k^{\scriptscriptstyle (k)}$ such that $T_j^{\scriptscriptstyle (k)}=T(x_j^{\scriptscriptstyle(k)},x_{j+1}^{\scriptscriptstyle(k)})$.  
Define
\begin{equation}\label{eq_limdef}   
\tilde\al(K):=\lim \sum_{j=1}^k \sqrt[3]{8\,V_2(T_j^{(k)})} 
\end{equation}
where the limit is taken over a sequence of subdivisions with 
\[\max_{i=1,\ldots,k} V_2(T_i^{ (k)}) \to\nolinebreak 0\]
as $k\to\infty$.  For smooth convex bodies in 
$\cK^2$, Blaschke showed that this limit always exists and that $\tilde\al(K)=\al(K)$.

If we choose a further subdivision point $y\in\partial K$ in a support triangle $T(x,z)$ of $K\in\cK^2$, we obtain support triangles $T(x,y)$ and $T(y,z)$ and the 
following elementary anti-triangle inequality holds
\begin{equation*}
\sqrt[3]{8\,V_2(T(x,z))}\ge \sqrt[3]{8\,V_2(T(x,y))}+\sqrt[3]{8\,V_2(T(y,z))}
\end{equation*}
(cf.~\cite[p.~38]{Blaschke} or \cite{Calabi:Olver:Tannenbaum}). This implies that
$\sum_{j=1}^k \sqrt[3]{8\,V_2(T_j^{\scriptscriptstyle (k)})}$ decreases as the subdivision is refined.
Consequently, the limit in \eqref{eq_limdef} exists and is independent of the sequence of subdivisions chosen and
\begin{equation*}
\tilde\al(K)=\inf \sum_{j=1}^k \sqrt[3]{8\,V_2(T_j^{(k)})}
\end{equation*}
where the infimum is taken over all subdivisions of $\partial K$. Thus $\tilde\al$ is well defined on $\cK^2$ and 
Leichtwei\ss\ \cite{Leichtweiss98} proved that $\tilde\al(K)=\al(K)$ for every  $K\in\cK^2$. This is also a simple consequence of Theorem \ref{thm_asa} for $n=2$. Indeed, $\tilde\al: \cK^2\to \R$ is equi-affine invariant and vanishes on lower dimensional sets. As an infimum of continuous functionals, $\tilde\al$ is upper semicontinuous. So we have only to show that $\tilde\al$ is a valuation.
Since $\max_{i=1,\ldots,k} V_2(T_i^{\scriptscriptstyle (k)}) \to0$ as $k\to\infty$,  we have  for every line $H$,
\[\tilde\al(K)= \tilde\al(K\cap H^+)+\tilde\al(K\cap H^-)\]
where $H^+$ and $H^-$ are the closed halfspaces bounded by $H$. It is not difficult to see that this implies that
$\tilde\al$ is a valuation. Thus Theorem \ref{thm_asa} for $n=2$ shows that  
\[\tilde\al(K)=c\, \al(K)\]
with a  constant $c\ge 0$ and a simple calculation for $K=B^2$ shows that $c=1$.

\section{Translation Invariant Valuations}

For translation invariant valuations on convex polytopes and on convex bodies, Hadwiger developed the basic theory. Many of the results are even valid in the setting of rational polytopes in $\Q^n$ and polytopes with integer coordinates (see, for example, \cite{BoeroeczkyLudwig_survey}).
Nevertheless, we will restrict our attention to convex polytopes and convex bodies in $\R^n$.

\subsection{The Canonical Simplex Decomposition}

For $0\le k\le n$, a $k$-dimensional simplex $S$ in $\R^n$ is the convex hull of $(k+1)$ affinely independent points $p_0, \dots, p_k\in\R^n$. We set $x_i:= p_i-p_{i-1}$ for $1\le i\le k$ and $x_0:=p_0$ and write $S=\langle x_0;x_1,\dots,x_k \rangle$. For $k=0$, we set $S:=\{x_0\}$.

\begin{lemma} A set $S$ is an $n$-dimensional simplex with vertices  $p_0, \ldots, p_n\in\R^n$ if and only  if
\begin{equation}\label{eq_csimplex}
S=\Big\{x_0+\sum_{i=1}^{n} r_{i} x_i:  1 \geq r_{1} \geq \ldots \geq r_n \geq 0\Big\}.
\end{equation}
Conversely, for  $x_{0}, \ldots, x_n\in \R^n$, the set defined in \eqref{eq_csimplex} is an $n$-dimensional simplex if $x_1, \dots, x_n$ are linearly independent.
\end{lemma}

\begin{proof}
Every point $x\in S$ is a convex combination of $p_0, \dots, p_n$, that is,  
\[x=\sum_{i=0}^{n} t_{i} p_{i}\] 
with $t_{i}\ge 0$ and $\sum_{i=0}^{n} t_{i}=1$. Setting $r_{i}=\sum_{j=i}^{n} t_{j}$, we have
\[x=x_0+\sum_{i=1}^{n} r_{i} x_{i}.\]
Hence, every point contained in the right side of \eqref{eq_csimplex} is in $S$.

Conversely, if $S$ is the set defined in \eqref{eq_csimplex}, then, setting  $p_{k}=\sum_{j=0}^{k} x_j$, we have
$S=[p_{0}, \ldots, p_{n}]$.
\end{proof}
\goodbreak

The following result is called the Hadwiger canonical simplex decomposition \cite[Section~1.2.6]{Hadwiger:V}. 

\begin{theorem}
Let $S:=\langle x_0;x_1,\dots,x_n\rangle$ be an $n$-dimensional simplex. Defining $\underline S_0:=\{x_0\}$, $\overline S_{n-k}:=\{x_0+\dots+x_n\}$, 
\[\underline S_k:=\big\langle x_0;x_1,\dots,x_k\big\rangle\quad\text{ and }\quad\overline S_{n-k}:=\big\langle x_0+\sum_{i=1}^k x_i;x_{k+1},\dots,x_n\big\rangle,\]
for $1\le k\le n-1$,
we have
\begin{equation*}
S=\bigsqcup_{k=0}^n \big( (1-t)\, \underline S_k+ t\, \overline S_{n-k}\big)
\end{equation*}
for $0<t<1$.\label{thm_canon}
\end{theorem}

\goodbreak
\begin{proof}
Setting
\[Q_k(t):= (1-t)\, \underline S_k+ t\, \overline S_{n-k},\]
we obtain by Lemma \ref{eq_csimplex} that
\begin{eqnarray*}
Q_k(t)&=&\Big\{(1-t)\big(x_0+\sum_{i=1}^k r_i\, x_i\big)+t\big(x_0+\sum_{i=1}^k x_i +\sum_{i=k+1}^n s_i\, x_i\big):\nonumber\\
&&\hspace{3cm} 1\ge r_1\ge \dots\ge r_k\ge 0, 1\ge s_{k+1}\ge \dots \ge s_n\ge 0\Big\}\nonumber\\
&=& \Big\{x_0+\sum_{i=1}^n t_i\,x_i\,: 1\ge t_1\ge \dots \ge t_k\ge t\ge t_{k+1}\ge\dots\ge t_n\ge0\Big\}.
\end{eqnarray*}
For $x\in S$, this implies that $x\in Q_k(t)$ for a suitable $k$. We have to show that the sets $Q_k(t)$ for $1\leq k \leq n-1$ have pairwise disjoint interiors. If  $x\in Q_i(t)\cap Q_k(t)$ for $i<k$, then $t_{i+1}=\dots=t_k=t$ and therefore $r_j=0$ and $s_j=1$ for $i+1\le j\le k$. It follows that $x\in\partial Q_i(t)$ and $x\in\partial Q_j(t)$. This completes the proof of the statement. \end{proof}

\begin{figure}
	\centering
	\tdplotsetmaincoords{55}{45}
	\begin{tikzpicture}[scale=3,line join=bevel,tdplot_main_coords]
		\tikzset{xzplane/.style={canvas is xz plane at y=#1}}
		\tikzset{yzplane/.style={canvas is yz plane at x=#1}}
		\tikzset{xyplane/.style={canvas is xy plane at z=#1}}
		
		\pgfmathsetmacro\tee{0.7}
		
	\draw [thick] (0,0,0) -- (2,0,0) -- (2,1.5,0);
		\draw [thick,densely dashed] (2,1.5,0) -- (0,0,0);
		\draw [thick] (0,0,0) -- (0.85,0,2) -- (2,0,0) -- cycle;
		\draw [thick] (0.85,0,2) -- (2,1.5,0);
		
		\draw [dashed] (2-2*\tee,0,0) -- (2,\tee*1.5,0);
		\draw (2,\tee*1.5,0) -- (2+0.85*\tee-2*\tee,0,2*\tee) -- (2-2*\tee,0,0);
		
		\draw [dashed] (2-2*\tee,0,0) -- (2-2*\tee,1.5-1.5*\tee,0) -- (2+0.85*\tee-2*\tee,1.5-1.5*\tee,2*\tee);
		\draw (2+0.85*\tee-2*\tee,1.5-1.5*\tee,2*\tee) -- (2+0.85*\tee-2*\tee,0,2*\tee);
				
		\draw [dashed] (0.85*\tee,0,2*\tee) -- (2+0.85*\tee-2*\tee,1.5-1.5*\tee,2*\tee);
		\draw (2+0.85*\tee-2*\tee,0,2*\tee) -- (0.85*\tee,0,2*\tee);
		
	\end{tikzpicture}
	\caption{Canonical simplex decomposition}
\end{figure}
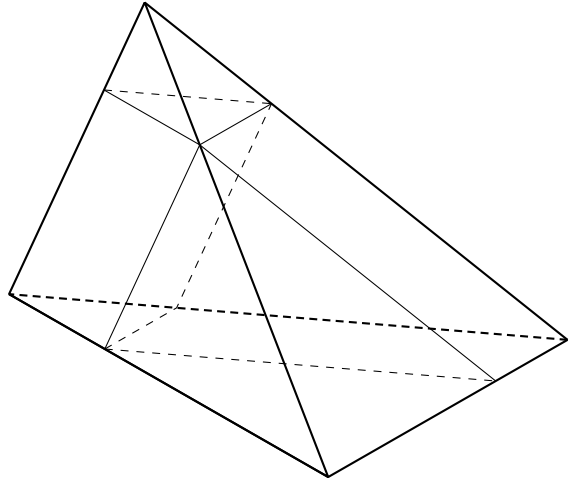

We say that a simplex $\langle x_0;x_1,\dots,x_n\rangle$ is \emph{orthogonal} if the vectors $x_1, \dots, x_n$ are pairwise orthogonal.  The following result is due to Hadwiger  \cite[Section 1.3.4]{Hadwiger:V}. 

\begin{lemma}
Let $z\in\R^n$ be given. 	
If $P\in\cP^n$ is $n$-dimensional,  then there are orthogonal simplices $S_1,\dots, S_{m}$, $S'_1,\dots, S'_{m'}$, each with a vertex at $z$, such that\label{lem_ortho}
	\[
	P\sqcup\bigsqcup_{i=1}^{m} S_i\sim \bigsqcup_{j=1}^{m'} S_j'.
	\]
\end{lemma}

\begin{proof}
The statement is easy to prove for $n=1$. Assume that it is true in $\cP(E)$ for every $(n-1)$-dimensional hyperplane $E$ and every $z_E\in E$. 

It suffices to prove the statement for an $n$-dimensional simplex $S$. 
Let $F$ be one of its facets whose affine hull $E$ does not contain $z$. 
Let $z_E$ be the closest point to $z$ in $E$. We  use the induction assumption for polytopes in $E$ with $z_E$ and obtain that  there are $(n-1)$-dimensional simplices $F_1, \dots, F_k,$ $F_1', \dots, F_{k'}'$, each with a vertex at $z_E$ such that
\[	F\sqcup\bigsqcup_{i=1}^{k} F_i\sim \bigsqcup_{j=1}^{k'} F_j'.\]
Setting $S_i:=[z,F_i]$ for $1\le i\le k$ and $S_j':=[z, F_j']$ for $1\le j\le k'$, we obtain the statement for $S$.
\end{proof}

\noindent
The question of whether every polytope in $\cP^n$ can be dissected into finitely many orthogonal simplices is open. Hadwiger conjectured that it is possible, and his conjecture has been proved for $n\le 5$ (see, for example, \cite{BKKS}).

\subsection{Valuations Vanishing on Orthogonal Cylinders}

We say that $P\in\cP^n$ is a convex \emph{orthogonal cylinder} if there are orthogonal, complementary subspaces $E$ and $F$ with $\dim E, \dim F \ge 1$ and  polytopes $P_E\subset E$ and $P_F\subset F$ such that $P=P_E+P_F$. 
Note that this class includes all polytopes that are not full-dimensional.

\begin{proposition}[Hadwiger]
If $\,\oZ\colon\cK^n\to\R$ is a continuous, translation invariant valuation that vanishes on convex orthogonal cylinders,   
then $\oZ$ is homogeneous of degree 1.\label{prop_onehom}
\end{proposition}

\begin{proof}
Let $S$ be an $n$-dimensional orthogonal simplex in $\R^n$ and $0<t<1$. In the canonical simplex decomposition,
\begin{equation}\label{eq_canon}
S=\bigsqcup\limits_{k=0}^n \big((1-t)\underline S_k+t\, \overline S_{n-k}\big),
\end{equation}
the simplices $\underline S_k$ and $\overline S_{n-k}$ are orthogonal and lie in orthogonal subspaces for each $1\le k\le n-1$. 	Hence $(1-t)\underline S_k+t\, \overline S_{n-k}$ is an orthogonal cylinder for $1\le k\le n-1$. Since $\oZ$ vanishes on convex orthogonal cylinders, we obtain 
\[\oZ((1-t)\underline S_k+t\, \overline S_{n-k})=0\]
for $1\le k\le n-1$ and we also see that $\oZ$  is simple. By \eqref{eq_finite_add}, it now follows from \eqref{eq_canon} that
\[
\oZ(S)= \oZ((1-t)S)+ \oZ(t\,S).
\]
Let $r,s>0$. Setting $\alpha(r):=\oZ(r\tilde S)$ with $S= (r+s)\tilde S$ and $t=r/(r+s)$, we obtain 
\[\alpha(r+s)=\alpha(r)+\alpha(s)\]
for all $r,s>0$. Since $\oZ$ is continuous, so is $\alpha:(0,\infty)\to \R$. It follows that $\alpha$ is a continuous Cauchy function. Hence $\alpha$ is linear. Thus,
\begin{equation}\label{eq_simplex_hom}
\oZ(t \,S)=t\oZ(S)
\end{equation}
for every $t>0$ and every orthogonal simplex $S$.

\goodbreak	
For $P\in\cP^n$,  by Lemma \ref{lem_ortho} there are orthogonal simplices such that
\[
P\sqcup\bigsqcup_{i=1}^m S_i\sim \bigsqcup_{j=1}^{m'} S_j'.
\]
Therefore, for every $t>0$, using that $\oZ$ is simple, we obtain
\[\oZ(t\,P)+\oZ\Big(\bigsqcup_{i=1}^m t\,S_i\Big)= \oZ\Big(\bigsqcup_{j=1}^{m'} t\,S_j'\Big)\]
and
\begin{equation*}
\oZ(t\, P) = \sum_{j=1}^{m'}\oZ\big( t\,S_j'\big)-\sum_{i=1}^m\oZ\big( t\,S_i\big)
	=t \sum_{j=1}^{m'}\oZ(S_j')-t\sum_{i=1}^m \oZ( S_i)
	= t\oZ(P),
	\end{equation*}
where \eqref{eq_finite_add} and \eqref{eq_simplex_hom} are used.
Hence, $\oZ$ is homogeneous of degree 1 on polytopes. Since $\oZ$ is continuous, this concludes the proof.
\end{proof}

\goodbreak
For $1\le \ell\le n$, we say that $P\in\cP^n$ is a convex \emph{$\ell$-cylinder} if there are  subspaces $E_1, \dots, E_\ell$ of $\R^n$ which are pairwise orthogonal and at least one-dimensional and convex polytopes $P_1\subset E_1$, \dots, $P_\ell\subset E_\ell$ such that $P=P_1+ \dots+ P_\ell$. We say that $C\subset \R^n$ is an $\ell$-cylinder if it can be dissected into finitely many convex $\ell$-cylinders. The following result was established by Hadwiger \cite[Section 1.3.7]{Hadwiger:V}.

\begin{theorem}
For $P\in\cP^n$ and each $1\le \ell\le n$, there is an $\ell$-cylinder $C_\ell$ such that 
\[m P= \bigsqcup_{\ell=1}^n \bigsqcup_{\tau\in \cT_{\ell,m}}\tau (C_\ell)\]
for every integer $m\ge 1$, 
where $\,\cT_{\ell, m}$ is a set of at most $\binom{m}{\ell}$ translations. \label{thm_icylinder}
\end{theorem}

\begin{proof}
It suffices to prove the statement for $P$ an $n$-dimensional simplex $S$. 
Let $S=\langle x_0;x_1,\dots,x_n\rangle$. For $i<j$, define the simplices 
$S_{i j}:=\langle x_i; x_{i+1},\dots,x_j\rangle$.
For $0<t<1$ and $1\le k <n$, set 
\[Q_k(t):=(1-t)\, S_{0k}+ t\, \Big(\sum_{i=0}^{k-1} x_i +S_{kn}\Big).\]
The canonical dissection into simplices for $m\,S$ and $t= \tfrac1m$ from Theorem \ref{thm_canon} gives
\[m S= m\,Q_0(\tfrac 1 m)\sqcup\cdots \sqcup\,m\,Q_n(\tfrac1 m).\]
We have $m\,Q_0(\tfrac 1 m)\approx S$ and  $m\,Q_n(\tfrac 1 m)\approx (m-1)\,S$ while 
\[m\,Q_\ell(\tfrac 1 m)\approx (m-1)\,S_{0\ell}+S_{\ell n}\] 
for $1\le \ell\le n-1$, where $\approx$ stands for equal up to translation. Applying the canonical simplex decomposition from Theorem \ref{thm_canon} to  $(m-1) S_{0 \ell}$, we obtain a decomposition into $\ell$-cylinders of the form
\[T_{\ell}\left(j_{1}, \ldots, j_{\ell}\right):=
S_{0 j_{1}}+S_{j_{1} n}+\dots+S_{j_{\ell-1} n}\]
for $1 \leq j_{1} <\dots < j_{\ell}=n$.
We use induction to show that each $T_{\ell}\left(j_{1}, \ldots, j_{\ell}\right)$ appears  $\binom{m}{\ell}$ times  in the decomposition. The statement is trivial for $m=1$. So, let $m>1$. The polytope $T_{\ell}\left(j_{1}, \ldots, j_{\ell}\right)$ appears when decomposing $(m-1) S$ and when decomposing $(m-1) S_{0 j_{\ell-1}}+S_{j_{\ell-1} n}$. By the induction assumption, it appears $\binom{m-1}{\ell}$ times in the first case and  $\binom{m-1}{\ell-1}$ times in the second case, which proves the claim. The $\ell$-cylinder $C_\ell$ is obtained as union of translates (with pairwise disjoint interiors) of the convex $\ell$-cylinders $T_{\ell}\left(j_{1}, \ldots, j_{\ell}\right)$ for $1 \leq j_{1} < \dots < j_{\ell}=n$.
\end{proof}

\subsection{The Homogeneous Decomposition Theorem}

The following result is fundamental in the theory of translation invariant valuations on convex bodies.

\begin{theorem}[McMullen]
If $\,\oZ:\cK^n\to \R$ is a continuous, translation invariant valuation, then 
\[\oZ=\oZ_0+\dots+\oZ_n\]
where $\,\oZ_j:\cK^n\to \R$  is a continuous, translation invariant and $j$-homogeneous valuation.
\label{thm_hom_decomp}
\end{theorem}

\noindent
We will prove the result under the additional assumption that $\oZ$ is simple. This version is due to Hadwiger. The general case was stated without proof by Hadwiger and proved by McMullen \cite{McMullen77}. We require the following proposition.

\begin{proposition}
If $\,\oZ:\cP^n\to \R$ is a simple, translation invariant valuation, then 
\[\oZ=\oZ_0+\dots+\oZ_n\]
where $\oZ_j:\cP^n\to \R$ for $0\le j\le n$ is a simple, translation invariant valuation that is homogeneous of degree $j$ with respect to multiplication by positive integers.
\end{proposition}

\goodbreak
\begin{proof}
Let $P\in\cP^n$. By Theorem \ref{thm_icylinder}, for $1\le j\le n$, there are $j$-cylinders $C_j$ such that
\begin{equation}\label{eq_mui}
\oZ(m\,P)=\sum_{j=1}^n \binom{m}{j} \oZ(C_j)
\end{equation}
for every $m\ge 1$. Note that $m\mapsto \oZ(m\, P)$ is a polynomial in $m$ of degree at most $n$.  We define $\oZ_j(P)$ as the coefficient of $m^j$ of this polynomial.

For $k,m\ge 1$, we obtain
\[\sum_{j=1}^n \oZ_j(k\,P)\,m^j =\oZ(k\,m\, P)=\sum_{j=1}^n \oZ_j(P)\,(k\,m)^j.\]
Therefore,
\[\oZ_j(k\,P)=k^j\,\oZ_j(P),\]
that is, $\oZ_j$ is homogeneous of degree $j$ with respect to multiplication with positive integers. To show that $\oZ_j$ is a valuation, it suffices to show that
\[\oZ_j(P\sqcup Q)=\oZ_j(P)+\oZ_j(Q).\]
This follows using \eqref{eq_mui} for $P\sqcup Q$, $P$ and $Q$ and comparing coefficients of $m^j$.
\end{proof}
\goodbreak

\begin{proof}[Proof of Theorem \ref{thm_hom_decomp} for simple valuations]
First, we show that for non-negative $\lambda\in\Q$, 
\[\oZ(\lambda P)= \sum_{j=1}^n \oZ_j(P)\,\lambda^j.\]
Indeed, let $\lambda=p/q$ with $p,q\in\N$. We have $q^j\oZ_j(\tfrac{1}{q}\,P)=\oZ_j(P)$ and
\[\oZ(\tfrac{p}{q}\,P)=\sum_{j=1}^n \oZ_j(\tfrac{1}{q}\, P)\,p^j= \sum_{j=1}^n \oZ_j(P) \big(\tfrac{p}{q}\big)^j.\]
So far, the valuations $\oZ_j$ are only defined on $\cP^n$. Note that the system of equations, 
\[\oZ(m\,P)=\sum_{j=1}^n \oZ_j(P)\, m^j\]
for $m=1, \dots, n$ with unknowns $\oZ_1(P), \dots, \oZ_n(P)$ has a unique solution, as the matrix is just the Vandermonde matrix. This gives us explicit representations,
\[\oZ_j(P)=\sum_{i=1}^n \alpha_{ij} \oZ(i\,P)\]
with suitable $\alpha_{ij}\in\R$ independent of $P$, which we use as definition of $\oZ_j$ on $\cK^n$. It is easy to see that the resulting functionals are continuous, translation invariant valuations that are homogeneous of degree~$j$. 
\end{proof}

For fixed $\bar{K}\in\cK^n$ and a given continuous, translation invariant valuation $\oZ$, Lemma \ref{lem_val_mink_add} shows that $K\mapsto \oZ_j(K +\bar{K})$ defines a continuous, translation invariant, $j$-homogeneous valuation on $\cK^n$. We may use this argument repeatedly and obtain the following theorem, where we call a function $\obZ:(\cK^n)^m\to\R$ \emph{symmetric} if
it is not changed when its arguments are permutated.

\begin{theorem}Let $1\le m\le n$. If $\,\oZ\colon\cK^n\to\R$ is  a continuous, translation invariant, $m$-homogeneous valuation, then there is a symmetric function $\obZ:(\cK^n)^m\to\R$ such that 
\[
\oZ(\lambda_1  K_1+\cdots +\lambda_k  K_k) = \sum_{\substack{i_1,\ldots, i_k\in \{0,\ldots,m\}\\i_1+\cdots+ i_k = m}} \binom{m}{i_1 \cdots i_k} \lambda_1^{i_1} \cdots \lambda_k^{i_k} \obZ(K_1 [i_1],\ldots,K_k [i_k])
\]
for every $k\ge 1$, every $K_1,\ldots,K_k\in\cK^n$ and every $\lambda_1,\ldots,\lambda_k\geq 0$.
Moreover,  $\obZ$ is  Minkowski additive in each variable and the map 
\[K\mapsto\obZ (K[j],K_{1},\ldots,K_{m-j})\]
is a continuous, translation invariant, $j$-homogeneous valuation for every $1\le j\le m$ and every $K_{1}, \dots, K_{m-j}\in\cK^n$.\label{thm_poly_sc}
\end{theorem}

\noindent
Here, we write $K[j]$ if $K$ appears $j$ times as an argument in $\obZ$ while a function $\oY: \cK^n\to \R$ is called \emph{Minkowski additive} if
\[\oY(K+L)= \oY(K)+\oY(L)\]
for every $K, L\in\cK^n$.  The special case $m=1$ in Theorem \ref{thm_poly_sc} leads to the following result. 

\begin{corollary}
If $\,\oZ\colon\cK^n\to\R$ is a continuous, translation invariant valuation that is homogeneous of degree 1, then  $\oZ$ is Minkowski additive. \label{cor_add} 
\end{corollary}

Theorem \ref{thm_hom_decomp} allows to reduce questions on continuous and translation invariant valuations to questions on such valuations with a given degree of homogeneity $j\in\{0, \dots, n\}$. It is easy to see that every continuous, translation invariant, and 0-homogeneous valuation is a multiple of the Euler characteristic. For the degrees of homogeneity $j=n$ and $j=n-1$, we mention (without proofs) the following results by Hadwiger \cite{Hadwiger:V} and McMullen \cite{McMullen80}.

\begin{theorem}[Hadwiger]
A functional $\oZ:\cP^n\to \R$ is  a translation invariant and $n$-homogeneous valuation if and only if  there is a constant $c\in\R$ such that 
$$\oZ(P) =c\, V_n(P)$$
for every $P\in\cP^n$. \label{thm_hugo_n}
\end{theorem}

\goodbreak

\begin{theorem}[McMullen]
A functional $\oZ:\cK^n\to \R$ is  a continuous, trans\-lation invariant, and $(n-1)$-homogeneous valuation  if and only if  there is $\zeta\in C(\sn)$ such that 
$$\oZ(K) =\int_{\sn} \zeta(y)\,\D S_{n-1}(K,y)$$
for every $K\in\cK^n$. The function $\zeta$ is uniquely determined up to addition of the restriction of a linear function.\label{thm_mcmullen_n-1}
\end{theorem}

\noindent
Here, $S_{n-1}(K, \cdot)$ is the surface area measure of $K$. Continuous, translation invariant, 1-homogeneous valuations were classified by Goodey and Weil  \cite{GoodeyWeil1984}. 

\goodbreak
While a complete classification of continuous, translation invariant valuations on $\cK^n$ is out of reach, Alesker \cite{Alesker01} proved the following result.

\begin{theorem}[Alesker]
For $\,0\le j \le n$, the space of linear combinations of the valuations
$$\big\{K\mapsto V(K[j], K_1, \dots, K_{n-j}):  K_1, \dots, K_{n-j}\in\cK^n\big\}$$
is dense in the space of continuous, translation invariant, $j$-homogeneous valuations. \label{thm_dense}
\end{theorem}

\noindent
Here, $V(K[j], K_1, \dots, K_{n-j})$ is the mixed volume of $K\in\cK^n$ taken $j$ times and $K_1, \dots, K_{n-j}\in\cK^n$ while the topology on the space of continuous, translation invariant valuations is induced by the norm 
\[\Vert \oZ \Vert:= \sup\{\vert \oZ(K)\vert: K\in\cK^n,\, K\subseteq B^n\}.\] 
Alesker's result confirms a conjecture by McMullen \cite{McMullen80} and is based on Alesker's so-called irreducibility theorem \cite{Alesker01}, which has further far-reaching consequences.

\goodbreak

For simple valuations,  the following complete classification was established by Klain \cite{Klain95} and Schneider \cite{Schneider:simple}.

\begin{theorem}[Klain \& Schneider]
A functional $\oZ:\cK^n\to \R$ is  a  continuous, translation invariant, simple valuation if and only if  there are $c\in \R$ and an odd function $\zeta\in C(\sn)$ such that 
$$\oZ(K) =\int_{\sn} \zeta(y)\,\D S_{n-1}(K,y)+c\,V_n(K)$$
for every $K\in\cK^n$. The function $\zeta$ is uniquely determined up to addition of the restriction of a linear function.\label{thm_klain_schneider}
\end{theorem}

\noindent
Klain \cite{Klain95} used his classification of simple valuations in his proof of the Hadwiger theorem. For an alternate proof of Theorem \ref{thm_klain_schneider}, see \cite{KusejkoParapatits}. 

A classification of \emph{weakly continuous} and translation invariant valuations on $\cP^n$ was obtained by McMullen \cite{McMullen1983}. Here,  a valuation is weakly continuous if it is continuous under parallel displacements
of the facets of polytopes.

\section{Rigid Motion Invariant Valuations}\label{sec:rigid}

The following rigid motion invariant, simple valuations are called \emph{Dehn invariants}. For $P\in\cP^3$ and $\zeta:[0,\infty)\to[0,\infty)$ a Cauchy function with $\zeta(\pi)=0$, set
\[\oD_\zeta(P):= \sum V_1(E) \,\zeta(\alpha_P(E))\]
where the sum is taken over all edges $E$ of $P$ and $\alpha_P(E)$ is the dihedral angle of $P$ at $E$. It is not difficult to see that $\oD_\zeta$ is a rigid motion invariant, simple valuation on $\cP^3$ and  that $\oD_\zeta$ vanishes on cubes. The regular tetrahedron $T$ in $\R^3$ has the dihedral angle $\alpha:=\arccos (1/3)$ at every edge, and the ratio $\alpha/\pi$ is irrational. Hence there are Cauchy functions with $\zeta(\alpha)\ne 0$. Hence $\oD_{\zeta}(T)\ne 0$ for every regular simplex $T$. Since $\oD_\zeta$ is a rigid motion invariant, simple valuation, it follows from \eqref{eq_finite_add} that $T$ is not equi-dissectable to any cube. This shows that Hilbert's Third Problem has a negative answer (for an introduction to Hilbert's Third Problem and the dissection theory of polytopes, see \cite{Boltianskii}). In general, $\oD_\zeta$ is far from being continuous. 

\goodbreak
A complete classification of rigid motion invariant and continuous valuations on $\cK^n$ was obtained by Hadwiger \cite[Section 6.1.10]{Hadwiger:V} in his 
celebrated classification theorem. 

\begin{theorem}[Hadwiger]
A functional $\oZ:\cK^n\to \R$ is  a continuous, translation and rotation invariant valuation if and only if  there are constants $c_0, \ldots, c_n\in\R$ such that 
\[\oZ(K) =  c_0 V_0(K) +\dots+ c_n V_n(K)\]
for every $K\in\cK^n$. \label{thm_hugo}
\end{theorem}

\noindent
Hadwiger \cite[Section 6.1.10]{Hadwiger:V} also obtained a complete classification of monotone increasing, translation and rotation invariant valuations by showing that any such valuation is a linear combination with non-negative coefficients of intrinsic volumes. McMullen \cite{McMullen1990} showed that every monotone increasing, translation invariant valuation is continuous. Hence the monotone version of Hadwiger's theorem is a simple consequence of Theorem \ref{thm_hugo}.

We present a variation of Hadwiger's original proof, which we got to know through lecture notes by Ulrich Betke. The main step is to prove the following result for simple valuations.

\goodbreak
\begin{proposition}
A functional $\oZ:\cK^n\to \R$ is  a continuous, translation and rotation invariant, simple valuation if and only if  there is  a constant $c\in\R$ such that 
\[\oZ(K) =  c \,V_n(K)\]
for every $K\in\cK^n$. \label{prop_hugo_simple}
\end{proposition}

\noindent
We first show how to deduce the Hadwiger theorem from this proposition and then describe its proof.
An alternate proof of the Hadwiger theorem is due to Dan Klain \cite{Klain95}. It can also be found in \cite{Klain:Rota} and \cite{Schneider:CB2}. Klain also uses the simple argument in the following subsection.

\goodbreak
\begin{proof}[Proof of the Hadwiger Theorem  using Proposition \ref{prop_hugo_simple}]

We use induction on the dimension $n$ and note that the statement is true for $n=1$ by Proposition~\ref{prop_trans1}. Assume that the statement is true in dimension $(n-1)$ and let $E$ be an $(n-1)$-dimensional linear subspace of $\R^n$. The restriction of $\oZ$ to $\cK(E)$ is a continuous, translation and rotation invariant valuation on $\cK(E)$. By the induction assumption, there are constants $c_0, \dots, c_{n-1}\in\R$ such that
\[\oZ= \sum_{j=0}^{n-1} c_j V_j\]
for every $K\in\cK(E)$. Define $\otZ:\cK^n\to \R$ by
\[\otZ:=\oZ-\sum_{j=0}^{n-1} c_j V_j\]
and note that $\otZ$ is a continuous, translation and rotation invariant valuation on $\cK^n$. Moreover, $\otZ$ is simple, as $\otZ$ vanishes on $\cK(E)$ and hence, because of its translation and rotation invariance, on all convex bodies contained in an affine hyperplane.  Using Proposition  \ref{prop_hugo_simple}, we obtain that there is a constant $c_n\in\R$ such that
\[\otZ(K) =c_n V_n(K)\]
for every $K\in\cK^n$. This concludes the proof. 
\end{proof}
\goodbreak

\subsection{A Characterization of the Mean Width}

For $K\in\cK^n$ and $u\in\sn$, the support function of $K$ in the direction $u$ is
\[h_K(u)=\sup_{x\in K}\langle x,u\rangle\]
and the width of $K$ in direction $u$ is $h_K(u)+h_K(-u)$. For the intrinsic volume $V_1$, which is defined in \eqref{eq_steiner}, we have
\[V_1(K)= \frac1{2\,\kappa_{n-1}} \int_{\sn} (h_K(u)+h_K(-u))\,\D \hm(u),\]
so $V_1(K)$ is proportional to the mean width of $K$.

We say that a convex body $M$ is a rotational Minkowski mean of $K\in\cK^n$ if there are rotations $\vartheta_1,\dots, \vartheta_m\in\son$ such that  
\[M= \frac1m \big( \vartheta_1 K + \dots + \vartheta_m K\big).\]
We require the following result due to Hadwiger \cite[Section 4.5.3]{Hadwiger:V}.

\begin{theorem}
For each $K\in\cK^n$, there exists a sequence of rotational Minkowski means of $K$ that converges to a centered ball. \label{thm_drehmittel}
\end{theorem}

\noindent
We remark that 
\[y\mapsto\int_{\son} h_{\vartheta K}(y)\,\D  \vartheta\]
is the support function of a centered ball associated with $K$, where integration is with respect to the Haar probability measure on $\son$. Hence the sequence from Theorem \ref{thm_drehmittel} can be obtained by a suitable discretization. For a complete proof, see, for example, \cite[Theorem 3.3.5]{Schneider:CB2}.

\goodbreak
\begin{theorem}[Hadwiger]
A functional $\oZ:\cK^n\to \R$ is  continuous, translation and rotation invariant, and Minkowski additive  if and only if  there is a constant $c\in\R$ such that 
\[\oZ(K) =c\, V_1(K)\]
for every $K\in\cK^n$. \label{thm_mean width}
\end{theorem}

\begin{proof}
For $K\in\cK^n$,  Theorem \ref{thm_drehmittel} implies that there exists a sequence $(K_j)$ of rotational Minkowski means of $K$ with $K_j\to r B^n$, where $r B^n$ is a centered ball of radius $r$, where $r$ depends on $K$. As every $K_j$ is of the form
\[K_j=\frac1m\big(\vartheta_1 K +\dots +\vartheta_m K\big)\]
with suitable rotations $\vartheta_1, \dots, \vartheta_m$,  we have $\oZ(K_j)= \oZ(K)$, as $\oZ$ is rotation invariant, Minkowski additive and homogeneous of degree 1. The continuity of $\oZ$ implies that 
\[\oZ(K)= \lim_{j\to\infty} \oZ(K_j)= \oZ(r B^n)= r\oZ(B^n).\] 
The first intrinsic volume, $V_1$, is continuous, translation and rotation invariant, and Minkowski additive. Hence we also have $V_1(K)= r\,V_1(B^n)$. Combined this gives $\oZ(K)= c\,V_1(K)$ with $c=  \oZ(B^n)/ V_1(B^n)$.
\end{proof}

\subsection{Proof of Proposition \ref{prop_hugo_simple}}

We use induction on the dimension $n$. The statement is true for $n=1$ by Proposition~\ref{prop_trans1}.

Let $n\ge2$ and assume that the statement is true for valuations defined on $\cK^k$ for $1\le k\le n-1$. Let $E$ and $F$ be orthogonal and complementary subspaces with $\dim E=k$ and $\dim F=n-k$. If we fix a convex body $K_F$ in $F$, then the functional 
\[K_E\mapsto \oZ(K_E+K_F)\]
is a valuation on $\cK(E)$ by Lemma \ref{lem_val_mink_add} which is easily seen to be simple and continuous. Moreover, it is invariant with respect to translations and rotations in $E$. Hence, by the induction assumption, there is $c(K_F)\in\R$ such that
\[\oZ(K_E+K_F)=c(K_F)\,V_k(K_E)\]
for every $K_E\in\cK(E)$. It follows from Lemma \ref{lem_val_mink_add} that $c: \cK(F)\to \R$ is a valuation, which is easily seen to be simple, continuous, translation and rotation invariant. Hence, by the induction assumption, there is $c_k\in \R$ such that
\begin{equation}\label{eq_EF}
\oZ(K_E+K_F)=c_k V_{n-k}(K_F)\,V_k(K_E)
\end{equation}
for every $K_E\in\cK(E)$ and $K_F\in\cK(F)$. Since $\oZ$ is translation and rotation invariant,  \eqref{eq_EF} holds for convex bodies $K_E$ and $K_F$ in any orthogonal and complementary subspaces $E$ and $F$ with $\dim E=k$. Evaluating on the unit cube, we obtain that $c_1=\cdots=c_{n-1}=:c$.

\goodbreak
Define $\otZ: \cK^n\to\R$ by
\[\otZ(K)= \oZ(K)- c\,V_n(K).\]
Note that $\otZ$ is a simple, continuous, translation and rotation invariant valuation that vanishes on orthogonal cylinders. By Proposition \ref{prop_onehom}, it is homogeneous of degree~1, and by Corollary~\ref{cor_add}, it is Minkowski additive. Using Theorem  \ref{thm_mean width}, we obtain that there is a constant $d\in\R$ such that
\[\otZ(K)= d\, V_1(K)\]
for every $K\in\cK^n$. Since $\otZ$ is simple and $n\ge 2$, we obtain that $d=0$ which concludes the proof of the theorem. \hfill$\Box$

\subsection{Valuations Invariant under Subgroups of $\on$}

For valuations invariant under the action of  subgroups of the orthogonal group, $\on$, Alesker \cite{Alesker00, Alesker07} obtained the following result.

\begin{theorem}[Alesker]
For a compact subgroup $G$ of \,$\on$, the linear space of continuous, translation and $G$ invariant valuations on $\cK^n$ is finite dimensional if and only if
$G$ acts transitively on $\sn$.
\end{theorem}

\noindent
As the classification of such subgroups $G$ is known, it is a natural task (which was already proposed in \cite{Alesker00}) to find bases for spaces of continuous, translation and $G$ invariant valuations for all such subgroups  (see  \cite{Alesker01, Bernig2009, Bernig:Fu} for some of the contributions).

\subsection{An Application of the Hadwiger Theorem}

The following result is a special case of the principal kinematic formula, which is due to Blaschke, Chern, Federer and Santal\'o (see \cite{Klain:Rota,SchneiderWeil}). We use integration with respect to the Haar measure on $\sont$, and the normalization is chosen so that on $\son$ we have the Haar probability measure, and translations are identified with $\R^n$ with the standard Lebesgue measure.

\goodbreak
\begin{theorem}
For $K, L\in\cK^n$,
\[\int_{\phi\in \sont} V_0(K\cap \phi L)\, \D\phi= \sum_{i=0}^n  \frac{\kappa_i\,\kappa_{n-i}}{\binom{n}{i}\,\kappa_n}\,V_i(K)\,V_{n-i}(L).\]
\end{theorem}

\begin{proof}
For $K,L\in\cK^n$, set 
\[\oZ(K,L):=\int_{\phi\in \sont} V_0(K\cap \phi L)\, \D\phi.\]
For $L\in\cK^n$, it is easy to see that 
$K\mapsto \oZ(K,L)$
is a continuous, translation and rotation invariant valuation on $\cK^n$. By Theorem \ref{thm_hugo}, there are  $c_0(L),\dots, c_n(L)\in\R$
such that
\[\oZ(K,L)= \sum_{j=0}^n c_j(L)V_j(K)\]
for every $K,L\in \cK^n$. For given $K\in\cK^n$, the functional $L\mapsto \oZ(K,L)$ is also a continuous, translation and rotation invariant valuation on $\cK^n$. Combined with the homogeneity of intrinsic volumes, it follows that also each of the maps $L\mapsto c_j(L)$ for $0\leq j \leq n$ is a continuous, translation and rotation invariant valuation. By Theorem~\ref{thm_hugo}, there are $c_{0j},\dots, c_{nj}\in\R$ such that
\[\oZ(K,L)= \sum_{i,j=0}^n c_{ij} V_i(K) \,V_j(L)\]
for every $K,L\in\cK^n$. The constants $c_{i j}$ can be determined by evaluating this formula for suitable convex bodies $K$ and $L$.
\end{proof}

\section{Valuations on Function Spaces}
We will extend and generalize valuations from (subsets of) the space of convex bodies to function spaces. Let $\funRn$ denote the space of all real-valued functions on $\R^n$.

\subsection{Definition}
One way to represent the set of convex bodies, $\cK^n$, within $\funRn$ is to assign to each body $K\in\cK^n$ its characteristic function $\chi_K\in\funRn$, which is given by
$$\chi_K(x)=\begin{cases}
1\quad &\text{for }x\in K\\
0\quad &\text{for }x\notin K.
\end{cases}$$
Using this embedding we assign to each functional $\oZ:\funRn\to\R$ the functional $\otZ:\cK^n\to\R$ by setting
$$\otZ(K):=\oZ(\chi_K)$$
for every $K\in\cK^n$.

\goodbreak
We now ask which conditions $\oZ$ needs to satisfy so that $\otZ$ is a valuation. By the definition of $\otZ$ we have
\begin{align*}
\oZ(\chi_K) + \oZ(\chi_L) &= \otZ(K)+\otZ(L)\\
&= \otZ(K\cap L) + \otZ(K\cup L)\\
&= \oZ(\chi_{K\cap L})+\oZ(\chi_{K\cup L})\\
&= \oZ(\chi_K \wedge \chi_L) + \oZ(\chi_K \vee \chi_L)
\end{align*}
for every $K,L\in\cK^n$ such that also $K\cup L\in\cK^n$. Here, $f\vee g$ and $f\wedge g$ denote the pointwise maximum and minimum of $f,g\in\funRn$, respectively. This motivates the following definition.

\goodbreak

\begin{definition}
\label{def:val_fun_space}
Let $X\subseteq \funRn$. A map $\oZ:X\to\R$ is a \emph{valuation} if
$$\oZ(f)+\oZ(g)=\oZ(f\wedge g)+\oZ(f\vee g)$$
for every $f,g\in X$ such that also $f\wedge g,f\vee g\in X$.
\end{definition}

\goodbreak

\noindent
Similarly, one may define valuations with values in any Abelian semigroup. Examples include vector-valued, matrix-valued, measure-valued valuations, and even Minkowski valuations, which are valuations with values in the space of convex bodies equipped with Minkowski addition.

We remark that there are various ways to represent convex bodies within the space of real-valued functions on $\R^n$. For many  
such representations, pointwise maxima and minima of functions correspond to 
unions and intersections of bodies.  
Repeating the above steps  for other 
embeddings of $\cK^n$ into $\funRn$   
leads to the same definition  
of valuations on (subsets of) $\funRn$.

\subsection{First Examples}

As for valuations on convex bodies, the simplest valuation $\oZ:X\to\R$ for any $X\subseteq \funRn$ is of the form $\oZ(f)=c$ with $c\in\R$. It is straightforward to check that this defines a valuation, although it may not be the most interesting one.

Next, let $X\subseteq\{f\in\funRn\colon \vert \int_{\R^n} f(x)\,\D x \vert < \infty\}$, where we consider Lebesgue integrals. Define $\oZ:X\to\R$ as
$$\oZ(f):=\int_{\R^n} f(x)\,\D x.$$
We claim that $\oZ$ is a valuation. Let $f,g\in X$.
Since $\R^n$ can be represented as the disjoint union
$$\R^n = \{f\geq g\} \sqcup  \{f<g\},$$
where $\{f\geq g\}:=\{x\in\R^n\colon f(x)\geq g(x)\}$ and $\{f<g\}$ is defined accordingly, we have
\begin{align}
\label{eq:int_is_a_val}
\begin{split}
\oZ(f) + \oZ(g)
&= \int_{\R^n} f(x) \,\D x + \int_{\R^n} g(x) \,\D x\\
&= \int_{\{f\geq g\}} f(x) \,\D x + \int_{\{f< g\}} f(x) \,\D x + \int_{\{f\geq g\}} g(x)\,\D x + \int_{\{f < g\}} g(x) \,\D x\\
&= \int_{\{f\geq g\}} (f\vee g)(x) \,\D x + \int_{\{f < g\}} (f\wedge g)(x) \,\D x\\
&\quad + \int_{\{f\geq g\}} (f\wedge g)(x) \,\D x + \int_{\{f < g\}} (f\vee g)(x) \,\D x\\
&= \int_{\R^n} (f\vee g)(x)\,\D x + \int_{\R^n} (f\wedge g)(x) \,\D x\\[4pt]
&= \oZ(f\vee g) + \oZ(f\wedge g).
\end{split}
\end{align}
Note that this valuation often plays the role of volume. For example, the Pr\'ekopa--Leindler inequality is a functional version of the Brunn--Minkowski inequality, where the usual $n$-dimensional volume on subsets of $\R^n$ is replaced by the integral of a function (see, for example, \cite{Gardner:BM}).

Further valuations on suitable function spaces are given by the map $f\mapsto f(\bar{x})$ for some fixed $\bar{x}\in\R^n$, the $p$th power of the $L_p$ norm, the moment matrix, the Fisher information matrix or the LYZ body. We refer to \cite{Baryshnikov_Ghrist_Wright_AiM_2013,Ludwig:SobVal, Ludwig:Fisher, Tsang:Lp, Tsang:Mink, Wang:SemiVal} for more details.

\subsection{A Short Overview of Results}

Defining analytic analogs of geometric concepts is, of course, not a new problem (see, for example, \cite[Chapter~9]{artstein_gianonopoulos_milman_II} and  \cite[Sections 9.5 and 10.15]{Schneider:CB2}). 
This section focuses on results where analogs of important valuations in geometry were found on function spaces.

The survey \cite{Ludwig:FunVal} describes some of the first results on valuations on function spaces. Among them are Tsang's characterization of valuations on $L_p$ spaces and $L_p$ stars \cite{Tsang:Lp,Tsang:Mink}, as well as the characterization of the moment matrix \cite{Ludwig:Covariance}, Fisher information matrix \cite{Ludwig:Fisher}, the LYZ body and its projection body \cite{Ludwig:SobVal}.
In the following, we will give a brief overview of some of the results not included in \cite{Ludwig:FunVal}.

\subsubsection{Quasi-concave Functions}
A real-valued function $f$ on $\R^n$ is called \emph{quasi-concave} if it is non-negative and if its superlevel sets,
$$\{f\geq t\}:=\{x\in\R^n\colon f(x)\geq t\},$$
are either empty or convex bodies for every $t>0$.  A natural approach to extending intrinsic volumes from convex bodies to quasi-concave functions is to integrate intrinsic volumes of the level sets of a given function with respect to suitable measures.  The following result was proved in \cite{Colesanti-Lombardi}.

\begin{theorem}
A map $\oZ$ is a rigid motion invariant, continuous and increasing valuation on the space of quasi-concave functions on $\R^n$ if and only if there are measures $\nu_i\in \cN_i$ for $0\leq i \leq n$ such that
$$\oZ(f)=\sum_{i=0}^n \int_{[0,\infty)} V_i(\{f\geq t\}) \,\D \nu_i(t)$$
for every quasi-concave $f:\R^n\to\R$.
\label{thm:class_qc}
\end{theorem}
Here, $\oZ$ is \emph{rigid motion invariant} if
$$
\oZ(f\circ \phi^{-1})=\oZ(f)
$$
for every quasi-concave $f:\R^n\to\R$ and every rigid motion $\phi: \R^n\to\R^n$. It is \emph{increasing} if $f\leq g$ pointwise implies $\oZ(f)\leq \oZ(g)$. For the precise definition of the classes $\cN_i$ of Radon measures on $[0,\infty)$ and for the topology used in Theorem~\ref{thm:class_qc}, we refer to \cite{Colesanti-Lombardi}.

A homogeneous decomposition theorem for valuations on quasi-concave functions that corresponds to Theorem~\ref{thm_hom_decomp} and a functional analog of Theorem~\ref{thm_hugo_n} were proved in \cite{Colesanti-Lombardi-Parapatits}. In the next section, we will discuss valuations on convex functions defined via superlevel sets.

\subsubsection{Convex Functions}
Recently, the first results on valuations on various spaces of convex functions were obtained.
Valuations on the space of coercive convex functions, $\fconv$, (see Section~\ref{subse:coercive_cvx_fcts} for the definition) were first classified in \cite{Cavallina:Colesanti}. A characterization of analogs of the Euler characteristic and the $n$-dimensional volume as $\sln$ invariant valuations was established in \cite{Colesanti-Ludwig-Mussnig-1}, and we will prove a special case of this result in the next section. See also \cite{Mussnig19,Mussnig21}.

In addition, Minkowski valuations were considered, and characterizations of functional analogs of the difference body and the projection body were obtained in \cite{Colesanti-Ludwig-Mussnig-2}. The former result is the following.
\begin{theorem}
A map $\oZ:\fconv\to\cK^n$ is a continuous, decreasing, translation invariant and $\sln$ covariant Minkowski valuation if and only there exists a continuous, decreasing $\zeta:\R\to[0,\infty)$ with $\int_0^\infty \zeta(t) \,\D t <+\infty$ such that
$$\oZ(u)= D [\zeta\circ u]$$
for every $u\in \fconv$.
\end{theorem}
\noindent
Here, $\oZ:\fconv\to\cK^n$ is \emph{decreasing} if $u\leq v$ pointwise implies $\oZ(v)\subseteq \oZ(u)$.  
It is \emph{$\sln$ covariant} if $\oZ(u\circ \phi^{-1})=\phi \oZ(u)$ for every $u\in\fconv$ and $\phi\in\sln$. The body $[\zeta\circ u]\in \cK^n$ is given by its support function for  $y\in \sn$ as
$$h([\zeta\circ u],y):=\int_0^{\infty} h(\{\zeta\circ u \geq t\},y)\,\D t,$$
and $DK:=K+(-K)$ denotes the \emph{difference body} of the convex body $K\in\cK^n$. For the definition of continuity of operators on $\fconv$, we refer to Section~\ref{subse:coercive_cvx_fcts}.

\subsubsection{Functions defined on $\sn$}
There are various results on valuations on spaces of real-valued functions on the unit sphere, $\sn$. Such functions are particularly interesting since they appear as radial functions of star bodies or more general star-shaped sets. Results on valuations on the set, $C(\sn)^+$, of positive, continuous functions on $\sn$ are equivalent to results on valuations on star bodies. The following result was established in \cite{Tradacete_Villanueva_JMAA,Villanueva:AiM} as  a result for valuations on star bodies.

\begin{theorem}
A map $\oZ:C(\sn)^+\to\R$ is a continuous and rotation invariant valuation if and only if there is a continuous function $\zeta:[0,\infty)\to\R$ such that
$$\oZ(f)=\int_{\sn} \zeta(f(x)) \, \D \hm(x)$$
for every $f\in C(\sn)^+$.
\end{theorem}

\noindent
Here, $\oZ$ is \emph{rotation invariant} if
$$\oZ(f\circ \phi^{-1})=\oZ(f)$$
for every $f\in C(\sn)^+$ and every $\phi\in \son$, and continuity of $\oZ$ is understood with respect to uniform convergence of functions.

Further results in this area include  characterizations of continuous valuations on $C(\sn)^+$ without additional invariance properties and of valuations on the space of Lipschitz functions on $\sn$ (see \cite{Colesanti-Pagnini-Tradacete-Villanueva:AiM, Colesanti-Pagnini-Tradacete-Villanueva:JFA, Tradacete_Villanueva_AiM}).

\subsubsection{$L_p$ spaces}
The Laplace transform was characterized as valuation on subspaces of $L^1(\R^n)$ in \cite{Li-Ma:JFA}.
A classification of Minkowski valuations on $L_p$ spaces for $p\ge 1$ was found in \cite{Ober}. It generalizes previous results by Tsang \cite{Tsang:Mink}.
Classifications of translation and $\sln$ invariant valuations on Sobolev spaces were obtained in \cite{MaSobolev}.

\subsubsection{Definable Functions}
A Hadwiger-type classification of valuations on definable functions was obtained in \cite{Baryshnikov_Ghrist_Wright_AiM_2013}.
Here, diverse and, in general, quite large spaces of bounded real-valued functions on $\R^n$ are called definable. 
For the precise definition, which uses so-called $o$-minimal systems, we refer to \cite{Baryshnikov_Ghrist_Wright_AiM_2013}.

The authors of \cite{Baryshnikov_Ghrist_Wright_AiM_2013} use integral geometry and approximations by step functions to find functional extensions of intrinsic volumes to definable functions. While the construction of these functionals is too technical to reproduce here, we mention that functional analogs of volume are of the form
$$f\mapsto \int_{\R^n} \zeta(f(x))\,\D x,$$
where
$\zeta:\R\to\R$ is a suitable function.

In addition, two different topologies on definable functions are introduced in \cite{Baryshnikov_Ghrist_Wright_AiM_2013}, and it is shown that the new functional versions of the intrinsic volumes are continuous. These functionals are then characterized as the only continuous, rigid motion invariant valuations on the function space.

Applications to sensor networks of the functional analogs of the Euler characteristic are described, for example, in \cite{Baryshnikov_Ghrist}).

\section{A First Classification of Valuations on Convex Functions}
One of the first results on valuations on convex functions \cite{Colesanti-Ludwig-Mussnig-1} 
characterizes functional analogs of the Euler characteristic and $n$-dimensional volume. It is a functional version of Theorem~\ref{thm_Blaschke}. 
We will prove a special case of the result.

\subsection{Functional Setting}
\label{subse:coercive_cvx_fcts}
We denote by
$$\fconvx:=\{u:\R^n\to(-\infty,+\infty]\colon u \text{ is l.s.c. and convex}, u \not\equiv +\infty\}$$
the space of extended real-valued, lower semicontinuous, convex, proper functions on $\R^n$. We will be mostly working on the subspace of coercive functions,
$$\fconv:=\left\{u\in\fconvx\colon \lim\nolimits_{|x|\to+\infty} u(x)=+\infty\right\}.$$
Observe that a function $u\in\fconvx$ is coercive if and only if there exist $a>0$ and $b\in\R$ such that
\begin{equation*}
u(x)\geq a |x| + b
\end{equation*}
for every $x\in\R^n$.

If $u\in\fconv$, then its sublevel sets
$$\{u\leq t\}:=\{x\in\R^n\colon u(x)\leq t\}$$
are compact, convex subsets of $\R^n$ for every $t\in\R$. We have $\{u\leq t\}=\varnothing$ for every $t<\min_{x\in\R^n} u(x)$ and $\{u\leq t\}\in \cK^n$ for every $t\geq \min_{x\in\R^n} u(x)$. Note that $u$ attains its minimum since it is lower semicontinuous and coercive. We will see that many operators can be generalized to $\fconv$ 
using sublevel sets.

It is easy to see that
\begin{equation}
\label{eq:lvl_sets_min_max}
\{u\vee v \leq t\}= \{u\leq t\} \cap \{v\leq t\}\quad \text{and}\quad  \{u\wedge v \leq t\}=\{u\leq t\} \cup \{v\leq t\}
\end{equation}
for every $u,v\in\fconv$ such that $u\wedge v\in\fconv$ and $t\in\R$. 
Also, observe that while the domain of $u$,
$$\dom u:=\{x\in\R^n\colon u(x)< +\infty\},$$
is convex, it need not be bounded or closed.

We equip $\fconvx$ and its subspaces with the topology associated to epi-convergence, also called $\Gamma$-convergence. For standard references on this topic, we refer to the books \cite{braides, dal_maso, RockafellarWets}.

On $\fconv$, we have the following simple description of epi-convergent sequences. Convergence of compact, convex sets is with respect to the Hausdorff metric, and a sequence of sets $K_j$ is convergent to the empty set if $K_j=\varnothing$ for every $j\geq j_0$ with some $j_0\in\N$.

\begin{definition}
\label{def:epi-conv_lvl_sets}
A sequence $u_k \in \fconv$ is epi-convergent to $u\in\fconv$ if  
$\{u_k\leq t\}$ converges to $\{u\leq t\}$ for every $t\not=\min_{x\in\R^n} u(x)$. In this case, we will simply write $u_k\to u$.
\end{definition}

\noindent
Note that in general $\{u_k\leq t_0\}$ does not converge to $\{u\leq t_0\}$ for $t_0=\min_{x\in\R^n} u(x)$. To see this, let $u\in\fconv$ be arbitrary and set $u_k(x):=u(x)+\frac 1k$ for $x\in\R^n$ and $k\in\N$.
It is easy to see that $u_k$ is epi-convergent to $u$, but 
$$\{u_k\leq t_0\}=\varnothing$$
for every $k\in\N$ whereas $\{u\leq t_0\}\not=\varnothing$.

While pointwise convergence is a good choice on many function spaces, it gives undesired results on $\fconv$. To see this, let
$$\ind_K(x):=\begin{cases}
0\,\quad &\text{for }x\in K\\
+\infty\quad &\text{for }x\not\in K,
\end{cases}$$
denote the (convex) indicator function of $K\in\cK^n$ and $B^n:=\{x\in\R^n\colon |x|\leq 1\}$ the unit ball in $\R^n$. The sequence of functions
$$u_k(x):=\ind_{\left(1-\frac 1k\right) B^n}(x)$$
for $x\in\R^n$ and $k\in\N$, is not converging pointwise to $\ind_{B^n}$ on $\{x\in\R^n\colon |x|=1\}$, whereas it is epi-convergent. Furthermore, epi-convergence implies pointwise convergence almost everywhere (see, for example, \cite[Theorem 7.17]{RockafellarWets}).

A simple consequence of Definition~\ref{def:epi-conv_lvl_sets} is the following result \cite[Lemma 8]{Colesanti-Ludwig-Mussnig-1}.

\begin{lemma}
For any sequence $u_k\in \fconv$ that is epi-convergent to some $u\in\fconv$, there exist $a>0$ and $b\in\R$ such that
\begin{equation}
\label{eq:uniform_cone}
u_k(x)\geq a|x|+b\quad \text{and} \quad u(x)\geq a|x|+b
\end{equation}
for every $x\in\R^n$ and $k\in\N$.\label{le:uniform_cone}
\end{lemma}

\subsection{Valuations on $\fconv$}
We define valuations on $\fconvx$ and its subspaces as in Definition~\ref{def:val_fun_space}, using, in addition, that
$$t\vee +\infty = +\infty \quad \text{and} \quad t\wedge +\infty = t$$
for any $t\in(-\infty,+\infty]$.

\goodbreak
Let $\oZ:\fconv\to\R$ be a valuation. We say that $\oZ$ is \emph{translation invariant} if 
$$\oZ(u\circ \tau^{-1})=\oZ(u)$$
for every $u\in\fconv$ and translation $\tau$ on $\R^n$. It is \emph{$\SL(n)$ invariant} if
$$\oZ(u\circ \vartheta^{-1})=\oZ(u)$$
for every $u\in\fconv$ and $\vartheta\in\SL(n)$. Continuity of $\oZ$ is understood with respect to epi-convergence.

Note that we use the inverse of transforms on $\R^n$ in the definitions above. This corresponds to applying the transforms to sublevel sets, that is, 
\begin{equation}
\label{eq:transform_lvl_sets}
\{u\circ \tau^{-1}\leq t\}=\tau \{u\leq t\} \quad \text{and} \quad \{u\circ \vartheta^{-1}\leq t\}=\vartheta \{u\leq t\}
\end{equation}
for every $u\in\fconv$, $t\in\R$, translation $\tau$ on $\R^n$ and $\vartheta\in\SL(n)$.
\goodbreak

In the following, we describe functional analogs of the Euler characteristic and $n$-dimensional volume. We set $\E^{-u(x)}=0$ if $u\in\fconv$ and $x\in\R^n$ are such that $u(x)=+\infty$.

\begin{lemma}
\label{le:affine_inv_vals}
The maps
\begin{equation}
\label{eq:exp_min_u}
u\mapsto \E^{-\min_{x\in\R^n} u(x)}
\end{equation}
and
\begin{equation}
\label{eq:int_exp_u}
u\mapsto \int_{\R^n} \E^{-u(x)}\,\D x
\end{equation}
define continuous, $\SL(n)$ and translation invariant valuations on $\fconv$.
\end{lemma}
\begin{proof}
We will first consider \eqref{eq:int_exp_u} and show that it is well-defined and finite. Observe that by Cavalieri's principle (or the layer cake principle combined with the Fubini--Tonelli theorem), we have
\begin{equation}
\label{eq:rewrite_level_sets}
\int_{\R^n} \E^{-u(x)}\,\D x = \int_0^{+\infty} V_n(\{\E^{-u}\geq s\})\,\D s = \int_{-\infty}^{+\infty} V_n(\{u\leq t\})\, \E^{-t}\,\D t
\end{equation}
for every $u\in\fconv$.

Consider first the case that $u\in\fconv$ is such that $u(x)\geq a|x|$ for $x\in\R^n$ with $a>0$. Note that this implies
$$\{u\leq t\}\subseteq \{x\colon a|x|\leq t\} = \tfrac{t}{a}B^n$$
for every $t\geq 0$. Combined with \eqref{eq:rewrite_level_sets}, this implies that 
\begin{align*}
0 \leq \int_{\R^n} \E^{-u(x)}\,\D x &= \int_{-\infty}^{+\infty}V_n(\{u\leq t\}) \E^{-t}\,\D t\\
&\leq \int_0^{+\infty} V_n\big(\tfrac{t}{a}B^n\big)\E^{-t}\,\D t\\
&= \frac{\kappa_n}{a^n} \int_0^{+\infty} t^{n}\E^{-t}\,\D t\\
&= \frac{\kappa_n n!}{a^n}.
\end{align*}
In the general case there exist $a>0$ and $b\in\R$ such that $u(x)\geq a|x|+b$ for $x\in\R^n$ and thus, since $u(x)-b\geq a|x|$ for $x\in\R^n$, we have
\begin{equation}
\label{eq:int_exp_u_est}
0\leq \int_{\R^n} \E^{-u(x)}\,\D x = \E^{-b} \int_{\R^n} \E^{-(u(x)-b)} \,\D x \leq \kappa_n \E^{-b} \frac{n!}{a^n}.
\end{equation}
The fact that \eqref{eq:int_exp_u} defines a valuation follows from \eqref{eq:rewrite_level_sets} and \eqref{eq:lvl_sets_min_max} combined with the fact that the $n$-dimensional volume, $V_n$, is a valuation on convex bodies. Similarly, using \eqref{eq:transform_lvl_sets}, it is easy to obtain $\SL(n)$ and translation invariance.
The valuation property can be proved analogous to \eqref{eq:int_is_a_val}.

It remains to show continuity. Let $u_k$ be a sequence in $\fconv$ that epi-converges to a function $u\in\fconv$. By Lemma~\ref{le:uniform_cone}, there exist $a>0$ and $b\in\R$ such that \eqref{eq:uniform_cone} holds.
Thus, similar to \eqref{eq:int_exp_u_est}, we obtain
$$0\leq \int_{\R^n} \E^{-u_k(x)}\,\D x \leq  \kappa_n \E^{-b} \frac{n!}{a^n}$$
for every $k\in\N$. Since the $n$-dimensional volume is continuous with respect to Hausdorff convergence, we have $V_n(\{u_k\leq t\})\to V_n(\{u\leq t\})$ as $k\to\infty$ for almost every $t\in\R$. 
Therefore, by the dominated convergence theorem, we obtain that
$$\lim_{k\to\infty} \int_{\R^n} \E^{-u_k(x)}\,\D x = \int_{-\infty}^{+\infty} V_n(\{u\leq t\})\, \E^{-t}\,\D t = \int_{\R^n} \E^{-u(x)}\,\D x,$$
and thus \eqref{eq:int_exp_u} is continuous.

In order to establish the properties of \eqref{eq:exp_min_u}, observe that
$$\E^{-\min_{x\in\R^n} u(x)}=\int_{\min_{x\in\R^n} u(x)}^{+\infty}\E^{-t}\,\D t = \int_{-\infty}^{+\infty} V_0(\{u\leq t\})\,\E^{-t} \,\D t$$
for every $u\in\fconv$. Thus, it is easy to see that the same arguments as above can be applied.
\end{proof}

The proof above demonstrates the simple strategy for finding a functional analog of an operator $\oZ:\cK^n\to\R$ by considering the map
$$u\mapsto \int_{-\infty}^{+\infty} \oZ(\{u\leq t\})\, \E^{-t} \,\D t$$
on $\fconv$, where we set $\oZ(\varnothing):=0$. Indeed, in many cases, this will define an operator on $\fconv$ with similar properties as the original operator $\oZ$. More generally, one can often replace $\E^{-t}\,\D t$ by a suitable measure $\D \mu(t)$. For some examples, see \cite{BobkovColesantiFragala,milman_rotem}.

\subsection{A Functional Analog of Blaschke's Result}
Similar to Blaschke's characterization of the Euler characteristic and the volume, Theorem~\ref{thm_Blaschke}, we can characterize their functional analogs, which we introduced in the last section.

We will prove the following result.
\begin{theorem}
\label{thm:functional_blaschke}
For $n\geq 2$, a map $\oZ:\fconv\to\R$ is a continuous, $\SL(n)$ and translation invariant valuation such that
\begin{equation}
\label{eq:vertical_condition}
\oZ(u+t)=\E^{-t}\oZ(u)
\end{equation}
for every $u\in\fconv$ and $t\in\R$, if and only if there are constants $c_0,c_n\in\R$ such that
\begin{equation}
\label{eq:functional_blaschke}
\oZ(u)=c_0\, \E^{-\min_{x\in\R^n} u(x)} + c_n \int_{\R^n} \E^{-u(x)} \,\D x
\end{equation}
for every $u\in\fconv$.
\end{theorem}
We remark that if we omit condition \eqref{eq:vertical_condition}, additional valuations will appear in a classification result \cite{Colesanti-Ludwig-Mussnig-1,Mussnig19,Mussnig21}. However, all known proofs of such results are considerably more involved than the proof of Theorem~\ref{thm:functional_blaschke}, which we present here.

Observe that \eqref{eq:vertical_condition} becomes more natural if we consider the corresponding result on the space of log-concave functions,
$$\LC=\{\E^{-u}\colon u\in\fconv\}.$$
The properties of valuations on $\LC$ are defined analogously to the corresponding properties for valuations on $\fconv$. The following result, which is equivalent to Theorem~\ref{thm:functional_blaschke}, is a consequence of \cite[Theorem 4]{Mussnig:LC}.

\goodbreak
\begin{theorem}
For $n\geq 2$, a map $\oZ:\LC\to\R$ is a continuous, $\SL(n)$ and translation invariant valuation such that
$$\oZ(s f)=s\oZ(f)$$
for every $f\in\LC$ and $s>0$, if and only if there are constants $c_0,c_n\in\R$ such that
$$\oZ(f) = c_0 \max_{x\in\R^n} f(x) + c_n \int_{\R^n} f(x) \,\D x$$
for every $f\in\LC$.
\end{theorem}

To prove Theorem~\ref{thm:functional_blaschke}, we will start with the following simple observation.
For $K\in\cKo^n$, where $\cKo^n=\{K\in\cK^n\colon 0 \in K\}$, and $x\in\R^n$, set
$$g_K(x):=\min \{\lambda >0 \colon x\in \lambda K\}.$$
Then $g_K$ is the \emph{gauge function} or \emph{Minkowski functional} of $K$. Since we have $0\in K$, it follows that $g_K\in\fconv$ with
$$\{g_K\leq t\}=tK$$
for every $t\geq 0$ and $\{g_K\leq t\}=\varnothing$ for every $t<0$.
\begin{lemma}
\label{le:oz_on_gauge}
Let $n\geq 2$. If $\oZ:\fconv\to\R$ is a continuous and $\SL(n)$ invariant valuation that satisfies \eqref{eq:vertical_condition}, then there are constants $c_0,c_n\in\R$ such that
$$\oZ(g_K+t)=c_0 \,\E^{-\min_{x\in\R^n} (g_K(x)+t)} + c_n \int_{\R^n} \E^{-(g_K(x)+t)}\,\D x$$
for every $K\in\cKo^n$ and $t\in\R$.
\end{lemma}
\begin{proof}
Since
$$g_K\vee g_L = g_{K\cap L}\quad\text{and}\quad g_K \wedge g_L = g_{K\cup L}$$
for every $K,L\in\cKo^n$ such that $K\cup L\in\cKo^n$ and convergence of $K_j$ to $K$ on $\cKo^n$ implies
$$g_{K_j}\to g_K,$$
it is easy to see that
$$\otZ(K):=\oZ(g_K)$$
defines a continuous valuation on $\cKo^n$. Furthermore, since $\oZ$ is $\SL(n)$ invariant, also $\otZ$ has this property.
Thus it follows from Theorem~\ref{thm:blaschke_pon} that there exist $\tilde{c}_0,\tilde{c}_n\in\R$ such that
$$\otZ(K)=\tilde{c}_0 V_0(K) + \tilde{c}_n V_n(K)$$
for every $K\in\cKo^n$. Note that this also implies that $\otZ$ is translation invariant, which is not evident from its definition. Next, observe that
$$\E^{-\min_{x\in\R^n} (g_K(x)+t)} = \E^{-t}=\E^{-t} V_0(K)$$
and
$$\frac{1}{n!} \int_{\R^n} \E^{-(g_K(x)+t)}\,\D x =\frac{\E^{-t}}{n!} \int_0^{\infty} V_n(s K) \E^{-s} \,\D s = \E^{-t} V_n(K)$$
for every $K\in\cKo^n$ and $t\in\R$, where we used a similar computation as in \eqref{eq:rewrite_level_sets}. The result now follows by setting $c_0:=\tilde{c}_0$ and $c_n:={\tilde{c}_n}/{n!}$, since
\begin{align*}
\oZ(g_K+t)&= \E^{-t} \oZ(g_K)\\
&= \E^{-t} \otZ(K)\\
&= \tilde{c}_0\, \E^{-t} V_0(K) + \tilde{c}_n \E^{-t} V_n(K)\\
&= c_0\, \E^{-\min_{x\in\R^n} (g_K(x)+t)} + c_n \int_{\R^n} \E^{-(g_K(x)+t)} \,\D x
\end{align*}
for every $K\in\cKo^n$ and $t\in\R$.
\end{proof}

We will use the following reduction principle, which was first established for valuations on Sobolev spaces in \cite{Ludwig:SobVal}. For simplicity, we will present the proof in dimension one and remark that its extension to higher dimensions is straightforward. See, for example, \cite[Lemma 5.1]{Mussnig19}.
\begin{lemma}
\label{le:reduction_1}
Let $\oZ_1,\oZ_2:\fconv\to\R$ be continuous, translation invariant valuations. If
\begin{equation}
\label{eq:coincide_gauge_function}
\oZ_1(g_P+t)=\oZ_2(g_P+t)
\end{equation}
for every polytope $P\in\cPo^n$ and $t\in\R$, then
$$\oZ_1(u)=\oZ_2(u)$$
for every $u\in\fconv$.
\end{lemma}
\begin{proof}[Proof for $n=1$]
Since the valuations $\oZ_1$ and $\oZ_2$ are continuous, it is enough to consider the case where $u\in\fconv$ is such that
$$u=\bigwedge_{i=1}^m w_i$$
with affine functions $w_i:\R^n\to\R$. In addition, we may also assume that the graph of $u$ has no edges parallel to the coordinate axis. We will use induction on the number $k$ of vertices of the graph of $u$.

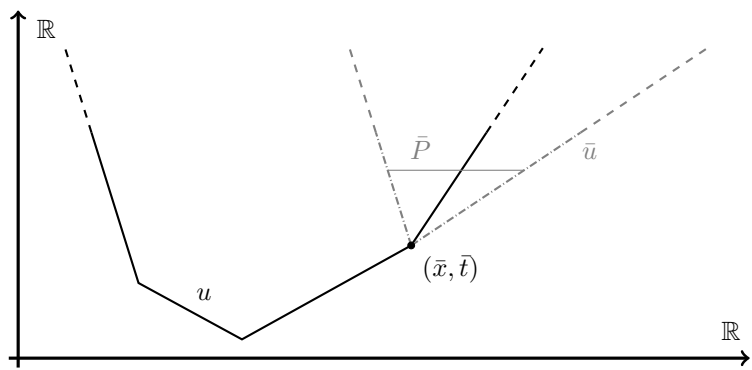
\begin{figure}[b]
	\centering
	\begin{tikzpicture}[scale=2.5]
        \draw[very thick,->] (-0.05-0.09,0.1) -- (3.8,0.1);
        \draw (3.7,0.14) node[anchor=south]{\fontsize{10}{10}\selectfont $\R$};
        \draw[very thick,->] (-0.09,-0.05+0.1) -- (-0.09,1.95);
        \draw (0.04-0.09,1.85) node[anchor=west]{\fontsize{10}{10}\selectfont $\R$};
        
        \draw[thick] (0.3,1.3)--(0.55,0.5)--(1.1,0.2)--(2,0.7)--(2.4,1.3);
        \draw[thick, dashed] (0.3,1.3)--(0.159375,1.75);
        \draw[thick, dashed] (2.4,1.3) -- (2.7,1.75);
        \draw (0.8,0.35) node[anchor=south west]{\fontsize{10}{10}\selectfont $u$};
        
        \draw[gray, densely dashdotted, thick] (2,0.7)--(2.9,1.3);
        \draw[gray, densely dashdotted, thick] (2,0.7)--(1.8125,1.3);
        \draw[gray, dashed, thick] (1.8125,1.3)--(1.671875,1.75);
        \draw[gray, dashed, thick] (2.9,1.3) -- (3.575,1.75);
        \draw[gray] (2.85,1.2) node[anchor=west]{\fontsize{10}{10}\selectfont $\bar{u}$};
        
        \draw[gray] (1.875,1.1)--(2.6,1.1);
        \draw[gray] (2.05,1.1) node[anchor=south]{\fontsize{10}{10}\selectfont $\bar{P}$};
        
        \draw[draw=black, fill=black] (2,0.7) circle (0.018);
        \draw (2,0.7) node[anchor=north west]{\fontsize{10}{10}\selectfont $(\bar{x},\bar{t})$};
    	\end{tikzpicture}
	\caption{Illustration of $u$ and $\bar{u}$ for the case $k=3$.}
	\label{fig:reduction}
\end{figure}

If $k=1$, then $u$ must be of the form
$$u(x)=g_P(x-x_0)+t$$
for some polytope $P\in\cPo^n$, $t\in\R$ and $x_0\in\R$. Thus, it follows from translation invariance and \eqref{eq:coincide_gauge_function} that $\oZ_1(u)=\oZ_2(u)$.

Assume now that the statement is true for $k-1$ and let the graph of $u$ have $k$ vertices. Denote by $(\bar{x},\bar{t})\in\R^2$ one of the highest vertices. It is easy to see that we can find a polytope $\bar{P}\in\cPo^n$ such that
$$\text{the graph of } u \wedge \bar{u} \text{ has } k-1 \text{ vertices},$$
where $\bar{u}\in\fconv$ is given by $\bar{u}(x):=g_{\bar P}(x-\bar{x})+\bar{t}$ for $x\in\R^n$. See Figure~\ref{fig:reduction}.

Since the graph of $u\vee \bar{u}$ has only one vertex, it follows from our induction assumption and the valuation property that
\begin{align*}
\oZ_1(u) &= \oZ_1(u\vee \bar{u})+\oZ_1(u\wedge \bar{u}) - \oZ_1(\bar{u})\\
&=\oZ_2(u\vee \bar{u})+\oZ_2(u\wedge \bar{u})-\oZ_2(\bar{u})\\
&=\oZ_2(u),
\end{align*}
which completes the proof.
\end{proof}

We can now proceed with the proof of our classification result.
\begin{proof}[Proof of Theorem~\ref{thm:functional_blaschke}]
If $\oZ:\fconv\to\R$ is as in \eqref{eq:functional_blaschke}, then it follows from Lemma~\ref{le:affine_inv_vals} that $\oZ$ is a continuous, $\SL(n)$ and translation invariant valuation. Furthermore, it is easy to see that it satisfies \eqref{eq:vertical_condition}.

Conversely, let $\oZ:\fconv\to\R$  be a continuous, translation and $\SL(n)$  invariant valuation that satisfies \eqref{eq:vertical_condition}. By Lemma~\ref{le:oz_on_gauge}, there exist $c_0,c_n\in\R$ such that
$$\oZ(g_K+t)=c_0\, \E^{-\min_{x\in\R^n} (g_K(x)+t)}+c_n \int_{\R^n} \E^{-(g_K(x)+t)} \,\D x$$
for every $K\in\cKo^n$ and $t\in\R$. Define now $\bar{\oZ}:\fconv\to\R$ as
$$\bar{\oZ}(u):=c_0\, \E^{-\min_{x\in\R^n} u(x)} + c_n \int_{\R^n} \E^{-u(x)}\,\D x.$$
By the first part of the proof, this is a continuous and translation invariant valuation such that $\bar{\oZ}(g_P+t)=\oZ(g_P+t)$ for every $P\in\cPo^n$ and $t\in\R$. The result now follows from Lemma~\ref{le:reduction_1}.
\end{proof}

\goodbreak
\subsection{Homogeneity}
For $u\in \fconv$ and $\lambda>0$, define $\lambda \odot u\in\fconv$ by
$$\lambda \odot u (x):=u\left(\frac x\lambda\right)$$
for $x\in\R^n$.
This definition is motivated by the fact that for $t\in\R$,
$$\{\lambda \odot u \leq t\}= \lambda \{u\leq t\}.$$
For $p\in\R$, an operator $\oZ:\fconv\to\R$ is \emph{horizontally $p$-homogeneous}  if
$$\oZ(\lambda\odot u)=\lambda^p\oZ(u)$$
for every $u\in\fconv$ and $\lambda>0$.

It is straightforward to see that $u\mapsto \E^{-\min_{x\in\R^n} u(x)}$ is horizontally $0$-homogeneous and that $u\mapsto \int_{\R^n} \E^{-u(x)} \,\D x$ is horizontally $n$-homogeneous. One might hope that similar to McMullen's decomposition theorem for valuations on $\cK^n$, Theorem~\ref{thm_hom_decomp}, every continuous and translation invariant valuation on $\fconv$ can be written as a sum of horizontally homogeneous valuations. However, such a result fails, and counterexamples were constructed in \cite[Theorem 1.2]{Colesanti-Ludwig-Mussnig-3}, where the following classes of functionals were studied. For $\zeta\in C_c(\R\times\R^n)$, the space of continuous functions with compact support on $\R\times \R^n$, consider
\begin{equation*}
\label{eq:counter_example_1}
u \mapsto \int_{\dom u} \zeta(u(x),\nabla u(x)) \,\D x,
\end{equation*}
which is well-defined since convex functions are differentiable almost everywhere on the interior of their domains. More general examples can be written as 
\begin{equation}
\label{eq:counter_example_2}
u \mapsto \int_{\R^{n}}\zeta(u(x),\nabla u(x))\,[\Hess u(x)]_i \, \D x,
\end{equation}
if in addition $u\in C^2(\R^n)$. Here $\Hess u(x)$ denotes the Hessian matrix of $u$ at $x\in \R^n$ and $[\Hess u(x)]_i$  the $i$th \emph{elementary symmetric function} of its eigenvalues for $0\leq i \leq n$. We remark that \eqref{eq:counter_example_2} can be extended to general $u\in\fconv$ where essentially $[\Hess u(x)]_i \, \D x$ is replaced by the so-called 
\emph{Hessian measures}.
The examples above are continuous and translation invariant. Still, due to their dependence on the gradient of the convex function, they cannot be decomposed into homogeneous terms  for all $\zeta\in C_c(\R\times\R^n)$ (see \cite{Colesanti-Ludwig-Mussnig-3} for details). Nevertheless, we will establish a functional analog of Theorem~\ref{thm_hom_decomp} in Section~\ref{se:hom_conv_fct}.

\section{Valuations on Super-Coercive Convex Functions}
\label{se:val_super_coerc}
To obtain a homogeneous decomposition theorem, we will restrict to the smaller space of super-coercive convex functions and shift our attention from sublevel sets to epi-graphs.

\subsection{Definitions and First Examples}
We consider the space of super-coercive convex functions,
$$\fconvs:=\Big\{u\in\fconvx\colon \lim_{|x|\to+\infty} \frac{u(x)}{|x|}=+\infty\Big\}.$$
Obviously, $\fconvs$ is a subspace of $\fconv$. Note that for differentiable $u\in\fconvs$, the property 
$$\lim_{|x|\to+\infty} \frac{u(x)}{|x|}=+\infty$$
implies that also
\begin{equation}
\label{eq:grad_infty}
\lim_{|x|\to+\infty} |\nabla u(x)| =  +\infty.
\end{equation}

The space of super-coercive convex functions is, in the following way, closely connected to the space of finite-valued convex functions,
$$\fconvf:=\{v:\R^n\to \R\colon v\text{ is convex}\}.$$
Recall that the Legendre transform is defined for $u\in\fconvx$ by
$$u^*(x):=\sup_{y\in\R^n} \big(\langle x,y \rangle - u(y) \big)$$
for $x\in\R^n$. By standard properties of the Legendre transform, we now have
\begin{equation}
\label{eq:fconvs_start_fconvf}
\{u^*\colon u\in\fconvs\}=\fconvf.
\end{equation}
This relation allows us to translate results for valuations on $\fconvs$ easily to results on $\fconvf$ and vice versa.

\goodbreak
A valuation $\oZ:\fconvs\to\R$ is  \emph{epi-translation invariant} if it is vertically translation invariant in addition to having the usual translation invariance, that is, if
$$\oZ(u\circ \tau^{-1}+\alpha)=\oZ(u)$$
for every $u\in\fconvs$, translation $\tau$ on $\R^n$ and $\alpha\in\R$. Note that this means that $\oZ$ is invariant under translations of the epi-graph of $u$ in $\R^{n+1}$, where the epi-graph of $u$ is given by
$$\epi u:=\{(x,t)\in\R^n\times \R\colon u(x)\leq t\}$$
and is a closed, convex subset of $\R^{n+1}$ for every $u\in \fconvx$.

\goodbreak
For $u,v\in\fconvs$, let 
$$(u\infconv v)(x):= \inf_{y\in\R^n} \big(u(x-y)+v(y)\big)$$
denote their \emph{infimal convolution} at $x\in\R^n$. Note that also $u\infconv v\in\fconvs$ and that
$$\epi (u\infconv v) = \epi u + \epi v,$$
where the addition on the right side is Minkowski addition of closed, convex sets in $\R^{n+1}$. 
The infimal convolution is also called \emph{epi-addition}. 
It naturally induces the following operation. 
For $\lambda >0$ and $u\in\fconvs$, define $\lambda\sq u\in\fconvs$ as
$$(\lambda \sq u)(x):=\lambda \sq u\left(\frac x\lambda \right)$$
for $x\in\R^n$. In addition, set $0\sq u := \ind_{\{0\}}$. It is easy to see that
$$\epi(\lambda \sq u)= \lambda \epi u$$
for every $\lambda >0$ and $u\in\fconvs$ and that
$$k\sq u = \underbrace{u \infconv \cdots \infconv u}_{k\text{ times}}$$
for every $k\in\N$ and $u\in\fconvs$.

For $p\in\R$, a valuation $\oZ:\fconvs\to\R$ is \emph{epi-homogeneous of degree $p$} if
$$\oZ(\lambda \sq u) = \lambda^p \oZ(u)$$
for every $u\in\fconvs$ and $\lambda>0$.
We will present two examples of continuous, epi-homogeneous, and epi-translation invariant valuations on $\fconvs$.

First, it is easy to see that the constant map $u\mapsto c$ for $c\in\R$ defines a continuous, epi-translation invariant valuation that is epi-homogeneous of degree $0$. In fact, it is the only valuation with these  properties (see \cite[Theorem 25]{Colesanti-Ludwig-Mussnig-4}).

\goodbreak
Next, let $\zeta \in C_c(\R^n)$, that is, $\zeta$ is continuous with compact support.  Consider the map
\begin{equation}
\label{eq:z_int_grad}
\oZ(u) = \int_{\dom u} \zeta(\nabla u(x)) \,\D x
\end{equation}
for $u\in\fconvs$.
Because of \eqref{eq:grad_infty}, it is at least plausible that $\oZ(u)$ is well-defined and finite for every $u\in\fconvs$. On the other hand, the example $u(x):=|x|$ shows that $\oZ$ defined by \eqref{eq:z_int_grad} is not a well-defined (finite) map on the larger space $\fconv$.

\goodbreak
We will state the following result from \cite[Proposition 20]{Colesanti-Ludwig-Mussnig-4} without proof.

\begin{lemma} For $\zeta\in C_c(\R^n)$, the map
$$u\mapsto \int_{\dom u} \zeta(\nabla u(x)) \,\D x$$
defines a continuous and epi-translation invariant valuation on $\fconvs$, which is epi-homogeneous of degree $n$.
\label{le:int_grad_is_a_val}
\end{lemma}

\goodbreak
\noindent
We remark that in the proof of this lemma, the map
$$v \mapsto \int_{\dom v^*} \zeta(\nabla v^*(x)) \,\D x$$
on $\fconvf$ is considered, 
which is well-defined by \eqref{eq:fconvs_start_fconvf}. This map can be written as
$$v\mapsto \int_{\R^n} \zeta(x) \,\D \MA(v;x)$$
for $v\in\fconvf$. Here, $\MA(v;\cdot)$, the \emph{Monge--Amp\`ere measure} of $v$, is a Radon measure on $\rn$, which can be defined as a continuous extension of the measure $\det(\Hess v(x))\,\D x$ from $C^2(\R^n)$ to $\fconvf$.

The valuation $\oZ$ defined in \eqref{eq:z_int_grad} can be seen as a further generalization of $n$-dimen\-sional volume on convex bodies to convex functions. Indeed, we have
\begin{equation}
\label{eq:int_grad_vol}
\oZ(\ind_K) = \int_K \zeta(0) \,\D x = \zeta(0) V_n(K)
\end{equation}
for every $K\in\cK^n$. See also Section~\ref{sec:current}.

\subsection{A Homogeneous Decomposition Theorem}
\label{se:hom_conv_fct}
We will prove a functional analog from \cite{Colesanti-Ludwig-Mussnig-4} of the homogeneous decomposition theorem, Theorem~\ref{thm_hom_decomp}.
\begin{theorem}
If $\oZ\colon\fconvs\to\R$ is a continuous, epi-translation invariant valuation, then
$$\oZ = \oZ_0+\cdots+\oZ_n$$
where $\oZ_j:\fconvs\to\R$ is a continuous, epi-translation invariant valuation that is epi-homogeneous of degree $j$.
\label{thm:mcmullen_fcts}
\end{theorem}

Similar to the proof of Theorem~\ref{thm:functional_blaschke}, we first show that the result is true on a restricted set of functions and then use a reduction argument. For $y\in\R^n$, define $\ell_y:\R^n\to\R$  by
$\ell_y(x):=\langle y,x\rangle$.

\begin{lemma}
\label{le:z_y_is_a_val}
Let $\,\oZ:\fconvs\to\R$ be a continuous, epi-translation invariant valuation. For every $y\in\R^n$, the map $\otZ_y:\cK^n\to\R$, defined by
$$\otZ_y(K):=\oZ(\ell_y+\ind_K),$$
is a continuous and translation invariant valuation.
\end{lemma}
\begin{proof}
Continuity and the valuation property are easy to obtain. For translation invariance, observe that
\begin{align*}
\ell_y(x)+\ind_{K+x_0}(x) &= \langle y,x \rangle + \ind_K(x-x_0)\\
&= \langle y,x-x_0 \rangle + \ind_K(x-x_0)+\langle y,x_0 \rangle\\
&= \ell_y(x-x_0)+\ind_K(x-x_0)+\langle y,x_0\rangle
\end{align*}
for every $K\in\cK^n$ and $x,x_0,y\in\R^n$. In other words, the epi-graph of $\ell_y + \ind_{K+x_0}$ is a translate of the epi-graph of $\ell_y+\ind_K$. Thus, by the epi-translation invariance of $\oZ$, we now obtain that
$$\otZ_y(K+x_0)=\oZ(\ell_y+\ind_{K+x_0})=\oZ(\ell_y+\ind_K)=\otZ_y(K)$$
for every $K\in\cK^n$ and $x_0,y\in\R^n$.
\end{proof}

Next, we use Theorem~\ref{thm_hom_decomp} to show that Theorem~\ref{thm:mcmullen_fcts} holds on functions of the form $\ell_y+\ind_K$ with $y\in\R^n$ and $K\in\cK^n$.
\begin{lemma}
\label{le:z_i}
If $\,\oZ:\fconvs\to\R$ is a continuous and epi-translation invariant valuation, then for every $u=\ell_y+\ind_K$ with $y\in\R^n$ and $K\in\cK^n$,
$$\oZ(u)=\oZ_0(u)+\cdots + \oZ_n(u)$$
where  $\oZ_i:\fconvs\to\R$ is a continuous and epi-translation invariant valuation such that
$$\oZ_i(\lambda \sq u) = \lambda^i \oZ(u)$$
for every $\lambda\geq 0$ and $u=\ell_y+\ind_K$ with $y\in\R^n$ and $K\in\cK^n$.
\end{lemma}
\begin{proof}
For $y\in\R^n$, define $\otZ_y:\cK^n\to\R$  by $\otZ_y(K):=\oZ(\ell_y+\ind_K)$. It follows from Lemma~\ref{le:z_y_is_a_val} and Theorem~\ref{thm_hom_decomp} that for every $y\in\R^n$ there exist continuous, translation invariant and $i$-homogeneous valuations $\otZ_{y,i}:\cK^n\to\R$ for $0\leq i\leq n$ such that
$$\otZ_y(K)=\sum_{i=0}^n \otZ_{y,i}(K)$$
for every $K\in\cK^n$. Thus,
$$\oZ(\lambda \sq (\ell_y+\ind_K))=\oZ(\ell_y+\ind_{\lambda K})=\otZ_y(\lambda K)=\sum_{i=0}^n \lambda^i\, \otZ_{y,i}(K)$$
for every $K\in\cK^n$, $y\in\R^n$ and $\lambda \geq 0$, where $0^0:=1$.  
Setting $\lambda:=j$ for every $0\le j\le n$, we therefore obtain
$$
\begin{pmatrix}
\oZ(0\sq (\ell_y+\ind_K)) \\ \vdots \\ \oZ(n\sq (\ell_y+\ind_K))
\end{pmatrix} =
\begin{pmatrix}
0^0 & \cdots & 0^{n}\\
\vdots & \ddots & \vdots\\
n^0 & \cdots & n^n
\end{pmatrix}
\begin{pmatrix}
\otZ_{y,0}(K)\\
\vdots\\
\otZ_{y,n}(K)
\end{pmatrix}
$$
for every $K\in\cK^n$ and $y\in\R^n$. The matrix in the equation is invertible since it is a Vandermonde matrix. Denoting its inverse by $(\alpha_{ij})_{0\leq i,j \leq n}$, we now have
$$\otZ_{y,i}(K)=\sum_{j=0}^n \alpha_{ij} \oZ(j \sq (\ell_y+\ind_K))$$
for $0\leq i \leq n$ and every $K\in\cK^n$ and $y\in\R^n$. Since the coefficients $\alpha_{ij}$ are independent of $K\in\cK^n$ and $y\in\R^n$, we may now define $\oZ_i:\fconvs\to\R$ as
$$\oZ_i(u):=\sum_{j=0}^n \alpha_{ij} \oZ(j\sq u)$$
for every $0\leq i \leq n$. It easily follows from the properties of $\oZ$ that also the functionals $\oZ_i$ for $0\leq i\leq n$ are continuous and epi-translation invariant valuations.

We now have
$$
\oZ_i(\ell_y+\ind_K)= \sum_{j=0}^n \alpha_{ij} \oZ(j\sq (\ell_y+\ind_{K}))= \otZ_{y,i}(K)
$$
and thus
$$\oZ_i(\lambda\sq(\ell_y+\ind_K))=\oZ_i(\ell_y+\ind_{\lambda K})=\otZ_{y,i}(\lambda K)= \lambda^i\, \otZ_{y,i}(K) = \lambda^i \oZ_i(\ell_y + \ind_K)$$
for every $0\leq i\leq n$, $K\in\cK^n$, $y\in\R^n$ and $\lambda>0$. Furthermore,
$$\oZ(\ell_y+\ind_K)=\otZ_y(K)=\sum_{i=0}^n \otZ_{y,i}(K) = \sum_{i=0}^n \oZ_i(\ell_y+\ind_K)$$
for every $K\in\cK^n$ and $y\in\R^n$, which completes the proof.
\end{proof}

We will now show that every continuous, epi-translation invariant valuation on $\fconvs$ is already determined by its values on a particular small set of functions.

\begin{lemma}
\label{le:reduction_2}
Let $\,\oZ:\fconvs\to\R$ be a continuous, epi-translation invariant valuation. If
\begin{equation}
\label{eq:oz_affine_zero}
\oZ(\ell_y+\ind_P)=0
\end{equation}
for every $y\in\R^n$ and $P\in\cP^n$, then $\oZ(u)=0$ for every $u\in\fconvs$.
\end{lemma}
\noindent
We present two approaches to prove this result. Similar to the proof of Lemma~\ref{le:reduction_1}, it follows from the continuity of $\oZ$ that we may reduce to the case that $u\in\fconvs$ is piecewise affine. Here, this means that there exist polytopes $P_1,\ldots,P_m\in\cP^n$ with pairwise disjoint interiors and affine functions $w_1,\ldots,w_m:\R^n\to\R$ such that
\begin{equation}
\label{eq:piecewise_affine}
u=\bigwedge_{i=1}^m (w_i+\ind_{P_i}),
\end{equation}
that is, $u$ is a piecewise minimum of affine functions restricted to disjoint polytopes.

\begin{proof}[Proof (First approach)]
We prove the result by induction on the number $m$ in \eqref{eq:piecewise_affine}. We start with the case $m=1$. Since there exist $y_1\in\R^n$ and $t\in\R$ such that $w_1=\ell_y+t$, we have $u=\ell_{y_1}+\ind_{P_1}+t$ and thus, by the epi-translation invariance of $\oZ$, it follows from \eqref{eq:oz_affine_zero} that $\oZ(u)=0$.

Assume that the statement is true for $m-1$ and let $u$ have $m$ components in the representation \eqref{eq:piecewise_affine}.
Without loss of generality, we may assume that there exist disjoint index sets $I_1,I_2\subset \{1,\ldots,m\}$ such that $0<|I_1|,|I_2|<m$ and such that the sets
$$\bigcup_{i\in I_1} P_i \quad \text{and}\quad  \bigcup_{i\in I_2} P_i$$
are convex (for a more detailed discussion of such a partition, we refer to \cite[Section~7.1]{Cavallina:Colesanti}).
Let $u_1,u_2\in\fconvs$ be defined as
$$u_1:=\bigwedge_{i\in I_1} (w_i+\ind_{P_i})\quad \text{and} \quad u_2:=\bigwedge_{i\in I_2} (w_i+\ind_{P_i}).$$
Clearly, $u=u_1\wedge u_2$ and, by our induction assumption, $\oZ(u_1)=\oZ(u_2)=0$. Furthermore, it is easy to see that if $\bar{u}:=u_1\vee u_2$, then there exist polytopes $\bar{P}_1,\ldots,\bar{P}_k\in\cP^n$ (which all have to be at most $(n-1)$-dimensional) with $k\leq m-1$ and affine functions $\bar{w}_1,\ldots,\bar{w}_k$ such that
$$\bar{u}=\bigwedge_{i=1}^k (\bar{w}_i+\ind_{\bar{P}_i}).$$
Using the induction assumption again, we see that also $\oZ(\bar{u})=0$. Thus, by the valuation property of $\oZ$,
$$\oZ(u)=\oZ(u_1)+\oZ(u_2)-\oZ(\bar{u})=0,$$
which completes the proof.
\end{proof}
\begin{proof}[Proof (Second approach)]
Similar to \eqref{eq_in_ex}, 
for every continuous valuation $\bar{\oZ}:\fconvs\to \R$ and $u_1,\ldots u_m\in\fconvs$ such that $u_1\wedge \ldots \wedge u_m\in\fconvs$, we have
$$\bar{\oZ}(u_1\wedge \cdots \wedge u_m)=\sum_{\varnothing \neq I \subset \{1,\ldots,m\}} (-1)^{|I|-1}\bar{\oZ}(u_I),$$
where $u_I:=\bigvee_{i\in I} u_i$. Thus, in order to show that $\oZ$ vanishes on functions of the form \eqref{eq:piecewise_affine}, it suffices to show that
$$\oZ\Big( \bigvee_{i\in I} (w_i+\ind_{P_i}) \Big) =0$$
for every $\varnothing \neq I \subset\{1,\ldots,m\}$. Since every such function $\bigvee_{i\in I} (w_i+\ind_{P_i})$ is again an affine function restricted to a polytope, the statement follows from \eqref{eq:oz_affine_zero} and the vertical translation invariance of $\oZ$.
\end{proof}

We now have all ingredients to prove the homogeneous decomposition theorem.
\begin{proof}[Proof of Theorem~\ref{thm:mcmullen_fcts}]
Let the valuations $\oZ_0,\ldots,\oZ_n:\fconvs\to\R$ be given by Lemma~\ref{le:z_i} and define $\bar{\oZ}:\fconvs\to\R$ as
$$\bar{\oZ}(u):=\oZ(u)-\sum_{i=0}^n \oZ_i(u).$$
Clearly, $\bar{\oZ}$ is a continuous and epi-translation invariant valuation. Furthermore, by the properties of the valuations $\oZ_i$ for $0\leq i \leq n$, we have
$$\bar{\oZ}(\ell_y+\ind_P)=0$$
for every polytope $P\in\cP^n$ and $y\in\R^n$. Thus, by Lemma~\ref{le:reduction_2},
$$\oZ(u)=\sum_{i=0}^n \oZ_i(u)$$
for every $u\in\fconvs$.

It remains to show that the valuation $\oZ_i$ is epi-homogeneous of degree $i$ for each $0\leq i\leq n$. For $\lambda \geq 0$ and $0\leq i\leq n$, set
$$\bar{\oZ}_{\lambda,i}(u)=\oZ_i(\lambda\sq u)-\lambda^i \oZ_i(u)$$
for $u\in\fconvs$. Note that $\bar{\oZ}_{\lambda,i}$ is a valuation on $\fconvs$. Using the same arguments as above, we obtain that $\bar{\oZ}_{\lambda,i}\equiv 0$, which shows that for each $0\leq i\leq n$, the valuation $\oZ_i$ is epi-homogeneous of degree $i$.
\end{proof}

\goodbreak
With an approach similar to that for Theorem~\ref{thm_poly_sc}, we obtain the following result by considering $u\mapsto \oZ_i(u \infconv \bar{u})$ for fixed $\bar{u}\in\fconvs$.

\begin{theorem}
Let $1\leq m\leq n$. If $\oZ\colon\fconvs\to\R$ is a continuous, epi-translation invariant valuation that is 
epi-homogeneous of degree $m$, then there is a symmetric function $\obZ:(\fconvs)^m\to\R$ such that 
$$
\oZ(\lambda_1 \sq u_1 \infconv \cdots \infconv \lambda_k \sq u_k) = \sum_{\substack{i_1,\ldots, i_k\in \{0,\ldots,m\}\\i_1+\cdots+ i_k = m}} \binom{m}{i_1 \cdots i_k} \lambda_1^{i_1} \cdots \lambda_k^{i_k} \obZ(u_1 [i_1],\ldots,u_k [i_k])
$$
for every $k\ge 1$, every $u_1,\ldots,u_k\in\fconvs$ and every $\lambda_1,\ldots,\lambda_k\geq 0$.
Moreover, $\obZ$ is epi-additive in each variable and the map 
\[u\mapsto\obZ (u[j],u_{1},\ldots,u_{m-j})\]
is a continuous, epi-translation invariant valuation on $\fconvs$ that is epi-homogeneous of degree~$j$ for $1\le j\le m$ and every $u_{1}, \dots, u_{m-j}\in\fconvs$.
\end{theorem}
\noindent
Here, a function $\oY: \fconvs\to \R$ is called \emph{epi-additive} if
$$\oY(u\infconv v)= \oY(u)+\oY(v)$$
for every $u,v\in\fconvs$. 
The special case $m=1$ in the previous theorem leads to the following result, which is a functional version of Corollary~\ref{cor_add}.

\begin{corollary}
If $\,\oZ\colon\fconvs\to\R$ is a continuous, epi-translation invariant valuation that is epi-homogeneous of degree 1, then $\oZ$ is epi-additive. 
\end{corollary}

\subsection{A Classification Result}
In this section, we will show that the valuations described in Lemma~\ref{le:int_grad_is_a_val} are indeed the only continuous and epi-translation invariant valuations on $\fconvs$ which are epi-homogeneous of degree $n$. The following result was established in \cite[Theorem~2]{Colesanti-Ludwig-Mussnig-4}.

\begin{theorem}
A map $\oZ:\fconvs\to\R$ is a continuous and epi-translation invariant valuation that is epi-homogeneous of degree $n$, if and only if there exists $\zeta\in C_c(\R^n)$ such that
\begin{equation}
\label{eq:oz_int_grad}
\oZ(u)= \int_{\dom u} \zeta(\nabla u(x))\,\D x
\end{equation}
for every $u\in\fconvs$. 
\label{thm:class_n-hom_fconvs}
\end{theorem}
\begin{proof}
For given $\zeta\in C_c(\R^n)$, it follows from Lemma~\ref{le:int_grad_is_a_val} that \eqref{eq:oz_int_grad} has the desired properties.

Conversely, let $\oZ:\fconvs\to\R$ be a continuous, epi-translation invariant valuation that is epi-homogeneous of degree $n$. For $y\in\R^n$, let $\otZ_y:\cK^n\to\R$ be defined by $\otZ_y(K):=\oZ(\ell_y+\ind_K)$. By Lemma~\ref{le:z_y_is_a_val}, the functional $\otZ_y$ is a continuous and translation invariant valuation. In addition, it is easy to see that $\otZ_y$ is homogeneous of degree $n$. Thus, it follows from Theorem~\ref{thm_hugo_n} that for every $y\in\R^n$ there exists a constant $\zeta(y)\in\R$ such that
\begin{equation}
\label{eq:otz_y_zeta_y_vn}
\oZ(\ell_y+\ind_K)=\otZ_y(K)=\zeta(y) V_n(K)
\end{equation}
for every $K\in\cK^n$. Furthermore, it follows from the continuity of $\oZ$ that $\zeta(y)$ continuously depends on $y\in\R^n$. This defines a continuous function $\zeta:\R^n\to\R$.

Next, we show that $\zeta$ has compact support. Assume on the contrary that there exists a sequence $y_k\in \R^n$ with
\begin{equation}
\label{eq:norm_y_k_infty}
\lim_{k\to\infty} |y_k|=+\infty
\end{equation}
but $\zeta(y_k)\neq 0$ for every $k\in\N$. By possibly restricting to a subsequence, we may assume without loss of generality that $\zeta(y_k)$ is positive for every $k\in\N$ and that there exists a vector $e\in\sn$ such that
$$\lim_{k\to\infty} \frac{y_k}{|y_k|}=e.$$
Let $B_k, B_\infty\in\cK^n$ be given by
$$B_k=\{x\in y_k^\perp\colon |x|\leq 1\},\quad B_\infty=\{x\in e^\perp\colon |x|\leq 1\}$$
and let $C_k\in\cK^n$ be defined as
\begin{equation}
\label{eq:def_c_k}
C_k=\Big\{x+t \frac{y_k}{|y_k|} \colon x\in B_k, t\in \Big[0,\frac{1}{\zeta(y_k)} \Big]\Big\}
\end{equation}
for $k\in\N$.
Observe that $C_k$ is an orthogonal cylinder and that
\begin{equation}
\label{eq:vn_ck}
V_n(C_k)= V_{n-1}(B_k) \frac{1}{\zeta(y_k)} = \frac{\kappa_{n-1}}{\zeta(y_k)}
\end{equation}
for $k\in\N$. Next, set $u_k:=\ell_{y_k}+\ind_{C_k}$ for $k\in\N$ and note that $u_k\in\fconvs$. It follows from \eqref{eq:norm_y_k_infty} and \eqref{eq:def_c_k} that
$u_k\to \ind_{B_\infty}$
as $k\to\infty$. Thus, the continuity of $\oZ$ combined with \eqref{eq:otz_y_zeta_y_vn} implies that
$$0=\oZ(\ind_{B_\infty})=\lim_{k\to\infty}\oZ(u_k).$$
On the other hand, it follows from \eqref{eq:otz_y_zeta_y_vn} and \eqref{eq:vn_ck} that
$$\oZ(u_k)=\zeta(y_k) V_n(C_k) = \kappa_{n-1}>0$$
for every $k\in\N$, which is a contradiction. Hence, we conclude that $\zeta$ has compact support.

It remains to show that \eqref{eq:oz_int_grad} holds. Define $\bar{\oZ}:\fconvs\to\R$ as
$$\bar{\oZ}(u):=\oZ(u)-\int_{\dom u} \zeta(\nabla u(x)) \,\D x.$$
By Lemma \ref{le:int_grad_is_a_val} and our assumptions on $\oZ$, the operator $\bar{\oZ}$ is a continuous and epi-translation invariant valuation. Furthermore, it follows from \eqref{eq:otz_y_zeta_y_vn} that
$$\bar{\oZ}(\ell_y+\ind_K) = \oZ_y(K)-\int_K \zeta(y) \, \D x = 0$$
for every $y\in\R^n$ and $K\in\cK^n$. Thus, Lemma~\ref{le:reduction_2} implies that
$$\bar{\oZ}(u)=0$$
for every $u\in\fconvs$, which completes the proof.
\end{proof}
  
\subsection{A Glimpse at the Current State of Research}
\label{sec:current}
As pointed out in \eqref{eq:int_grad_vol}, for any $\zeta\in C_c(\R^n)$, the operator
$$u\mapsto \int_{\dom u} \zeta(\nabla u(x))\,\D x$$
can be seen as a functional analog of the $n$-dimensional volume on $\fconvs$. This interpretation is further supported by Theorem~\ref{thm:class_n-hom_fconvs}, which (up to the assumption of continuity)  is a functional version of Theorem~\ref{thm_hugo_n}. In the following, we restrict to the rotation invariant case, where we say that a valuation $\oZ:\fconvs\to\R$ is \emph{rotation invariant} if
$$\oZ(u\circ \vartheta^{-1}) = \oZ(u)$$
for every $u\in\fconvs$ and $\vartheta\in \son$. Define
$$\oZZb{n}{\alpha}(u):=\int_{\dom u} \alpha(\vert \nabla u(x) \vert ) \,\D x$$
for $u\in \fconvs$ and $\alpha\in C_c([0,\infty))$.

It is a consequence of Theorem~\ref{eq:int_grad_vol} that, for each $0\leq j \leq n-1$, there exists a continuous, epi-translation invariant valuation $\oZZb{j}{\alpha}:\fconvs\to\R$ that is epi-homogeneous of degree $j$ such that
\begin{equation}
\label{eq:functional_steiner_formula}
\oZZb{n}{\alpha}(u\infconv r \sq \ind_{B^n}) = \sum_{j=0}^n r^{n-j}\kappa_{n-j} \oZZb{j}{\alpha}(u)
\end{equation}
for every $u\in\fconvs$ and $r\geq 0$. Observe that \eqref{eq:functional_steiner_formula} corresponds to the classical Steiner formula \eqref{eq_steiner} where we have replaced the $n$-dimensional volume with $\oZZb{n}{\alpha}$ and where now $\ind_{B^n}$ plays the role of the unit ball.

In many ways, the functionals $\oZZb{j}{\alpha}$ behave like the classical intrinsic volumes. First, it follows from the rotation invariance of $\oZZb{n}{\alpha}$ and the radial symmetry of $\ind_{B^n}$ that also $\oZZb{j}{\alpha}$ is rotation invariant for every $0\leq j \leq n-1$. Next, since
$$\ind_{K}\infconv r\sq \ind_{B^n} = \ind_{K + r B^n},$$
it follows from \eqref{eq_steiner}, \eqref{eq:int_grad_vol} and \eqref{eq:functional_steiner_formula} that
$$\oZZb{j}{\alpha}(\ind_K) = \alpha(0) V_j(K)$$
for every $0\leq j\leq n$ and $K\in\cK^n$. Last but not least, the functionals $\oZZb{j}{\alpha}$ are characterized by a Hadwiger-type theorem. The version that is stated here follows from \cite[Theorem 1.3]{Colesanti-Ludwig-Mussnig-5} and \cite[Theorem 1.4]{Colesanti-Ludwig-Mussnig-7}.

\begin{theorem}
Let $n\geq 2$. A functional $\oZ:\fconvs\to\R$ is a continuous, epi-translation and rotation invariant valuation if and only if there are functions $\alpha_0,\ldots,\alpha_n\in C_c([0,\infty))$ such that
$$\oZ(u)=\oZZb{0}{\alpha_0}(u)+\cdots+\oZZb{n}{\alpha_n}(u)$$
for every $u\in\fconvs$.
\label{thm:hadwiger_fcts}
\end{theorem}
\noindent
Theorem~\ref{thm:hadwiger_fcts} is a functional analog of the Hadwiger theorem, Theorem~\ref{thm_hugo}, and shows that the valuations $\oZZb{j}{\alpha}$ clearly play the role of the intrinsic volumes on $\fconvs$.

In \cite{Colesanti-Ludwig-Mussnig-5}, a different approach and notation are used. The functionals there take the form
\begin{equation}
\label{eq:functional_intrinsic_volume}
u\mapsto \int_{\R^n} \zeta(|\nabla u(x)|) [\Hess u(x)]_{n-j} \,\D x
\end{equation}
if in addition $u\in C^2(\R^n)$, where $\zeta\colon(0,\infty)\to\R$ has bounded support and might have a certain singularity at $0^+$. It was later shown in \cite[Theorem 1.4]{Colesanti-Ludwig-Mussnig-7} that the continuous extensions of \eqref{eq:functional_intrinsic_volume} to $\fconvs$ coincide with the functionals $\oZZb{j}{\alpha}$, that are considered here, where $\zeta$ and $\alpha$ are connected via an integral transform.

There are many open questions concerning  functional intrinsic volumes and related functionals. Current research topics  include characterization results, particularly for further groups of transformations,  and the program to obtain results in the integral geometry of function spaces and to establish inequalities for the newly defined functionals. We refer to \cite{Alesker:ConvFcts,Knoerr:JFA,Knoerr:JDG, Colesanti-Ludwig-Mussnig-6, Colesanti-Ludwig-Mussnig-7, Colesanti-Ludwig-Mussnig-8} for some recent results.

\subsection*{Acknowledgement}
 M. Ludwig was supported, in part, by the Austrian Science Fund (FWF):  P~34446, and F. Mussnig was supported by the Austrian Science Fund (FWF): J~4490.

\end{document}